\documentclass{amsart}
\usepackage{amsmath, amssymb, enumerate}
\usepackage{amsrefs}
\usepackage{hyperref} 
\usepackage{enumerate}
\usepackage{verbatim}
\usepackage{esint}
\usepackage{color}

\theoremstyle{plain}
\newtheorem{theorem}{Theorem}
\newtheorem*{piecetheorem}{Theorem \ref{piecethm}}
\newtheorem{lemma}[theorem]{Lemma}
\newtheorem{prop}[theorem]{Proposition}
\newtheorem{corollary}[theorem]{Corollary}
\newtheorem{conj}[theorem]{Conjecture}

\theoremstyle{definition}
\newtheorem{definition}[theorem]{Definition}

\theoremstyle{remark}
\newtheorem{remark}[theorem]{Remark}

\newcommand{\e}{\epsilon}
\newcommand{\R}{\mathbb{R}}

\newcommand{\G}{\mathbb{G}}

\renewcommand{\H}{\mathbb{H}}

\newcommand{\N}{\mathbb{N}}
\newcommand{\cl}{\overline}
\newcommand{\ti}{\textit}
\newcommand{\RR}{\mathbb{R}}
\newcommand{\HH}{\mathcal{H}}
\newcommand{\LIP}{\textnormal{LIP}}
\newcommand{\obar}[1]{\overline{#1}}

\DeclareMathOperator{\dist}{\textup{\text{dist}}}
\DeclareMathOperator{\diam}{\textup{\text{diam}}}
\DeclareMathOperator{\length}{\textup{\text{length}}}

\DeclareMathOperator{\supp}{\textup{\text{supp}}}
\DeclareMathOperator{\Lip}{\textup{\text{Lip}}}
\DeclareMathOperator{\Tan}{\textup{Tan}}
\DeclareMathOperator{\Hdim}{\textup{Hdim}}

\numberwithin{equation}{section}
\numberwithin{theorem}{section}

\begin{document}


\title[Lipschitz maps from PI spaces to Carnot groups]{Lipschitz and bi-Lipschitz maps from PI spaces to Carnot groups}
\author{Guy C. David}
\address{Department of Mathematical Sciences, Ball State University, Muncie, IN 47306}
\email{gcdavid@bsu.edu}

\author{Kyle Kinneberg}
\address{Department of Mathematics, Rice University, Houston, TX 77005 \newline \indent National Security Agency}
\email{kekinneberg@gmail.com}

\thanks{The first author was partially supported by the National Science Foundation under Grants No. DMS-1664369 and DMS-1758709.}

\date{\today}

\begin{abstract}
This paper deals with the problem of finding bi-Lipschitz behavior in non-degenerate Lipschitz maps between metric measure spaces. Specifically, we study maps from (subsets of) Ahlfors regular PI spaces into sub-Riemannian Carnot groups. We prove that such maps have many bi-Lipschitz tangents, verifying a conjecture of Semmes. As a stronger conclusion, one would like to know whether such maps decompose into countably many bi-Lipschitz pieces. We show that this is true when the Carnot group is Euclidean. For general Carnot targets, we show that the existence of a bi-Lipschitz decomposition is equivalent to a condition on the geometry of the image set.
\end{abstract}

\maketitle

\section{Introduction}

Let $X$ be a metric space of Hausdorff dimension $Q >0$, and $Y$ another metric space. Suppose $f \colon X \rightarrow Y$ is a Lipschitz map that is non-degenerate, in the sense that $f(X)$ has positive $Q$-dimensional Hausdorff measure in $Y$. This would occur, for example, if $f$ were a bi-Lipschitz homeomorphism, but there is no reason that $f$ need be bi-Lipschitz in general. Indeed, a weaker possibility that would yield the same conclusion is that $X$ contains a subset of positive measure on which $f$ is bi-Lipschitz. One can now ask whether this must be the case, and it is not easy to see why not.

To put it more generally, this paper is about the following question. Under what circumstances must there be some bi-Lipschitz behavior in the non-degenerate Lipschitz mapping $f$?

There are different forms that such bi-Lipschitz behavior could take. One might ask any of the following questions about the map $f$. They are ordered so that, under reasonable assumptions, a ``yes'' answer to one question implies a ``yes'' answer to those above it.
\begin{enumerate}[\normalfont (i)]
\item\label{existstangent} Must $f$ have a single tangent (or even ``weak tangent'') which is bi-Lipschitz? That is, must there be a sequence of scales along which one can ``zoom in'' on $f$ and pass to a limit map which is bi-Lipschitz?
\item\label{manytangents} Must $f$ have bi-Lipschitz tangents on a set of positive measure in $X$? In other words, must there be a positive measure set of points at which this ``zooming in'' yields a bi-Lipschitz map?
\item\label{bilippiece} Must there be a set $E \subset X$ of positive measure on which the restriction of $f$ is bi-Lipschitz?
\item\label{quantbilippiece} Must there be a set $E \subset X$ of positive measure on which the restriction of $f$ is bi-Lipschitz, with quantitative bounds on both the size of $E$ and the bi-Lipschitz constant of $f$ that are independent of $f$?
\end{enumerate}

For general $X$ and $Y$, the answer to all of these questions may be negative, as we discuss below. There are, however, a number of cases where one is guaranteed to find bi-Lipschitz behavior in Lipschitz mappings.

If $X=[0,1]^n\subset\RR^n$ and $Y=\RR^m$, then it is a classical fact, a consequence of Rademacher's theorem, that the answer to \eqref{bilippiece} is ``yes'', i.e., the map $f$ must have pieces on which it is bi-Lipschitz \cite[Lemma 3.2.2]{Fe69}. This was extended to the case in which $X=[0,1]^n\subset\RR^n$ and $Y$ is an arbitrary metric space by Kirchheim \cite{Ki94}. Kirchheim's extension relies on defining and proving a suitable version of Rademacher's theorem (Lipschitz maps are differentiable almost everywhere) for metric-space-valued Lipschitz mappings.

The quantitative question \eqref{quantbilippiece} was first studied in connection with problems on singular integrals and uniform rectifiability. A positive answer was established in Euclidean spaces by David \cite{Da88} and Jones \cite{Jo88}, and for mappings from $[0,1]^n\subset \RR^n$ to arbitrary metric spaces by Schul \cite{Sc09}. Related quantitative results appear for sub-Riemannian Carnot groups in \cites{Me13, Li15} and for certain metric manifolds in \cite{GCD16}. A general framework for addressing question \eqref{quantbilippiece} appears in \cite{Se00}.

In this paper we concern ourselves with the qualitative questions \eqref{existstangent}, \eqref{manytangents}, and \eqref{bilippiece}. In contrast to most of the previously studied cases discussed in the preceeding paragraph, which have domains $X$ that are Euclidean or Carnot, we are interested in the case where the domain of the mapping is allowed to be a quite general metric space supporting a notion of calculus, while the target $Y$ will be a Euclidean space or Carnot group. Carnot groups are a class of nilpotent Lie groups equipped with sub-Riemannian metrics that generalize many geometric and analytic features of Euclidean spaces. We discuss them more precisely in section \ref{Carnotsection}.

The domain spaces $X$ that we will study are the so-called ``PI spaces'': doubling metric measure spaces supporting a Poincar\'e inequality in the sense of Heinonen--Koskela \cite{HK98}. These spaces play a central role in the modern theory of analysis on metric spaces \cites{Ch99, He01, HKST15} and, relevant to our purposes, are known to support a form of differential calculus for Lipschitz functions by a theorem of Cheeger \cite{Ch99}. We will discuss PI spaces and Cheeger's theory in detail in section \ref{PILD}.

It is important to note that when one allows domains $X$ that are more general than Euclidean space, all the questions \eqref{existstangent} through \eqref{quantbilippiece} may have negative answers. For the simplest counterexample to \eqref{bilippiece}, one may consider $X$ to be $[0,1]$ equipped with the metric $d(x,y) = |x-y|^{1/2}$, giving it Hausdorff dimension $2$, and let $f$ be a Lipschitz mapping from $X$ onto the standard square $[0,1]^2$. The classical Hilbert space-filling curve gives a construction of such a map $f$. This example is discussed further in \cite{Me13} and \cite{GCD16}. There are also more elaborate examples due to David and Semmes \cite[Propositions 14.5 and 14.43]{DS97}, with even stronger analytic properties, which also provide negative answers to all four questions \eqref{existstangent} through \eqref{quantbilippiece}.

Even when the domain $X$ is a PI space, which implies strong restrictions on both the geometry of $X$ and the behavior of Lipschitz functions on it, questions \eqref{existstangent} through \eqref{quantbilippiece} may have negative answers. Indeed, there is an example due to Laakso \cite{La99} of an Ahlfors $Q$-regular PI space $X$ and a Lipschitz mapping from $X$ onto itself such that no tangent of $f$ is bi-Lipschitz, i.e., for which the answer to \eqref{existstangent} (and hence, to the three other questions) is ``no''. For more discussion on when these ``Lipschitz implies bi-Lipschitz behavior'' results can fail, see \cite[Chapter 5]{Se01} or \cite[Section 8]{GCD16}.

The point of discussing these counterexamples is to explain that, in order to hope for positive answers even when $X$ is a PI space, one must put some constraint on the target $Y$. Our constraint in this paper will be that $Y$ is a sub-Riemannian Carnot group.

\subsection{Semmes' questions and the main results}
In \cite{Se01}, Semmes, having digested the aforementioned counterexamples, makes a number of specific conjectures, related to the above questions, about the structure of Lipschitz mappings from PI spaces into Euclidean spaces and, more generally, into Carnot groups.

\begin{conj}[\cite{Se01}, Conjecture 5.3]\label{conjpiece}
Let $X$ be an Ahlfors $Q$-regular PI space, $A \subset X$ a subset, and $\G$ a sub-Riemannian Carnot group. Suppose $f \colon A \rightarrow \G$ is a Lipschitz mapping such that $\mathcal{H}^Q(f(A)) > 0$. Then there is a subset $E \subset A$ such that $\mathcal{H}^Q(E) > 0$ and $f$ is bi-Lipschitz on $E$.
\end{conj}

\begin{conj}[\cite{Se01}, Conjecture 5.9]\label{conjtangent}
Let $X$ be an Ahlfors $Q$-regular PI space, $A \subset X$ a subset, and $\G$ a sub-Riemannian Carnot group. Suppose $f \colon A \rightarrow \G$ is a Lipschitz mapping such that $\mathcal{H}^Q(f(A)) > 0$. Then there is a ``weak tangent'' of $f$ that is bi-Lipschitz.
\end{conj}

In this paper, we prove three theorems pertaining to Semmes' conjectures. First, we verify Conjecture \ref{conjtangent} completely. In fact, we show that such mappings have many bi-Lipschitz tangents.

\begin{theorem}\label{tangentthm}
Let $X$ be an Ahlfors $Q$-regular PI space, $A \subset X$ a subset, and $\G$ be a sub-Riemannian Carnot group. Suppose $f \colon A \rightarrow \G$ is a Lipschitz mapping such that $\mathcal{H}^Q(f(A)) > 0$.

Then there is a subset $E \subset A$, with $\mathcal{H}^Q(f(A \setminus E)) = 0$, such that, at each point $x \in E$, there exists $(\hat{X}, \hat{x}, \hat{f}) \in \Tan(A,x,f)$ with $\hat{f}$ bi-Lipschitz.
\end{theorem}

We discuss tangents and the $\Tan$ notation in Section \ref{tangentsec}. As one might expect, $f$ having a bi-Lipschitz tangent implies that $f$ has a bi-Lipschitz weak tangent, and so Theorem \ref{tangentthm} indeed proves Conjecture \ref{conjtangent}. In other words, for non-degenerate mappings from PI spaces to Carnot groups, the answers to questions \eqref{existstangent} and \eqref{manytangents} are both ``yes''.

We also prove Conjecture \ref{conjpiece}, and thereby give a positive answer to question \eqref{bilippiece}, under certain structural assumptions on the image set $f(A) \subset \G$.

\begin{theorem} \label{piecethm}
Let $X$ be an Ahlfors $Q$-regular PI space, $A \subset X$ a compact subset, and $\G$ be a sub-Riemannian Carnot group. Suppose $f \colon A \rightarrow \G$ is a Lipschitz mapping such that $\mathcal{H}^Q(f(A)) > 0$, and let $\nu = \mathcal{H}^Q|_{f(A)}$. Then the following are equivalent.

\begin{enumerate}[\normalfont (i)]
\item\label{piecei} There are countably many compact subsets $A_i \subset A$ such that $\nu(f(A \setminus \cup_i A_i)) = 0$
and $f$ is bi-Lipschitz on each $A_i$.
\item\label{pieceii} At $\nu$-a.e. point $y \in f(A)$, each tangent measure $\hat{\nu} \in \Tan(\nu,y)$ is comparable to the restriction of $\mathcal{H}^Q$ to a Carnot subgroup of $\G$.
\item\label{pieceiii} At $\nu$-a.e. point $y \in f(A)$, the support of each tangent measure $\hat{\nu} \in \Tan(\nu,y)$ is a connected subset of $\G$.
\end{enumerate}
\end{theorem}

Finally, we prove that Conjecture \ref{conjpiece} holds, and so the answer to question \eqref{bilippiece} is positive, if the target Carnot group $\G$ is a Euclidean space. 

\begin{theorem}\label{Rnpiecethm}
Let $X$ be an Ahlfors $Q$-regular PI space, $A \subset X$ a compact subset, and $N \in \N$. Suppose $f \colon A \rightarrow \R^N$ is a Lipschitz mapping such that $\mathcal{H}^Q(f(A)) > 0$. Then there are countably many subsets $A_i \subset A$ such that $\mathcal{H}^Q(f(A \setminus \cup A_i)) = 0$ and $f$ is bi-Lipschitz on each $A_i$.
\end{theorem}

We note that in Theorems \ref{piecethm}\eqref{piecei} and \ref{Rnpiecethm}, one obtains a conclusion slightly stronger than requested in question \eqref{bilippiece} above. Namely, the image of $f$ is covered, up to measure zero, by images of sets on which $f$ is bi-Lipschitz. This stronger version is typical in these type of theorems, appearing for example in \cite[Lemma 3.2.2]{Fe69} and in some of the other results discussed above.

Note also that in Theorems \ref{piecethm} and \ref{Rnpiecethm} we assume that $A\subset X$ is compact, whereas Semmes' Conjecture \ref{conjpiece} asks about arbitrary sets. However, the relevant special cases of Conjecture \ref{conjpiece} follow immediately from applying these results to compact subsets of $\overline{A}$.

\begin{remark}
An astute reader of Semmes' book \cite{Se01} may notice that Semmes' notion of a PI space, on the face of it, is different than the one used here (introduced in subsection \ref{PILD}), which is now more standard. However, Semmes' definition and ours are proved to be equivalent in \cite{KR04}, and so Theorems \ref{tangentthm}, \ref{piecethm}, and \ref{Rnpiecethm} do indeed address Semmes' Conjectures \ref{conjpiece} and \ref{conjtangent}.
\end{remark}

The terminology and notation appearing in these three theorems will be discussed in Sections \ref{backsec} and \ref{tangentsec}. The latter section is devoted to a discussion of various types of tangent objects that we need: tangents of metric spaces, tangents of measures, and tangents of mappings between metric spaces. In Section \ref{LQsec} we prove Theorem \ref{tangentthm}, showing further that every every tangent of $f$ at a point $x \in E$ is a Lipschitz quotient onto its image. Then, in Section \ref{DCsec}, we establish some general criteria for a mapping to have bi-Lipschitz pieces in terms of its tangent maps. Finally, in Section \ref{piecesec} we prove Theorem \ref{piecethm}, and in Section \ref{Rnpiecesection} we prove Theorem \ref{Rnpiecethm}.

It is likely that Theorem \ref{Rnpiecethm} can be proven by an extension of the method used to prove Theorem \ref{piecethm}. However, we instead present a short proof based on results from \cite{GCD14} and \cite{BL15}. This proof works only in the case of $\RR^n$ targets and not for the general case of Carnot group targets.

\section{Background} \label{backsec}

\subsection{Metric measure spaces} \label{basicdefs}
We denote metric spaces by $(X,d)$, or simply $X$ if the metric is understood, writing $B(x,r)$ for the open ball of radius $r$ centered at $x\in X$ and $\overline{B}(x,r)$ for the closed ball. If $\lambda>0$, then we write $\lambda X$ for the metric space $(X,\lambda d)$. A metric space $(X,d)$ is said to be \ti{metrically doubling} if there is a constant $C$ such that every ball of radius $r >0 $ in $X$ can be covered by at most $C$ balls of radius $r/2$.

A \ti{metric measure space} $(X,d,\mu)$ is a separable metric space $(X,d)$ equipped with a Radon measure $\mu$. We say that the measure $\mu$ is \ti{doubling} if there is a constant $C$ such that
$$\mu(B(x,2r)) \leq C \mu(B(x,r))$$
for all $x \in X$ and $r>0$. The measure $\mu$ is \ti{pointwise doubling} if
$$ \limsup_{r \rightarrow 0} \frac{\mu(B(x,2r))}{\mu(B(x,r))} < \infty$$
for $\mu$-a.e. $x\in X$.

For $Q>0$, we will consider $Q$-dimensional Hausdorff measure $\mathcal{H}^Q$ on a metric space $(X,d)$. To fix our normalizations, we define Hausdorff measure as follows. For $E \subset X$ a subset and $\delta >0$, let
$$\mathcal{H}^Q_{\delta}(E) = \inf\{ \sum_i r_i^Q : E \subseteq \cup_i B(x_i,r_i) \text{ with each } r_i < \delta \},$$
and define $\mathcal{H}^Q(E) = \lim_{\delta \rightarrow 0} \mathcal{H}^Q_{\delta}(E)$. We denote the Hausdorff dimension of a metric space $X$ by $\Hdim(X)$.

A complete metric space $X$ is \ti{Ahlfors $Q$-regular}, for $Q>0$, if there is a constant $C>0$ such that
$$ C^{-1}r^Q \leq \mathcal{H}^Q(\overline{B}(x,r)) \leq Cr^Q $$
for all $x\in X$ and $0<r\leq \diam(X)$. In this case, $\mathcal{H}^Q$ is a doubling Radon measure.

\subsection{David--Semmes regular and Lipschitz quotient mappings}
The main classes of mappings considered in this paper are the Lipschitz and bi-Lipschitz mappings. However, there are two intermediate classes that play key roles.

\begin{definition}[\cite{DS97}, Definition 12.1]\label{DSregulardef}
Let $X,Y$ be metric spaces and $f \colon X \rightarrow Y$ a Lipschitz mapping. We say that $f$ is \ti{David--Semmes regular} (with constant $C\geq 1$) if, for every ball $B=B(y,r)$ in $Y$, we can cover $f^{-1}(B)$ in $X$ with at most $C$ balls of radius $Cr$.
\end{definition}

\begin{definition}[\cite{BJLPS99}]\label{LQdef}
Let $X,Y$ be metric spaces and $f \colon X\rightarrow Y$ a mapping. We say that $f$ is a \ti{Lipschitz quotient mapping} if there are constants $C,c>0$ such that
\begin{equation}\label{LQdefeqn}
B(f(x),cr) \subseteq f(B(x,r)) \subseteq B(f(x),Cr)
\end{equation}
for all $x\in X$ and $r>0$.
\end{definition}
Note that the second inclusion in \eqref{LQdefeqn} simply says that the mapping is $C$-Lipschitz. The constant $c$ is called the \ti{co-Lipschitz} constant of $f$. Since precise constants do not matter much below, we will typically say that a mapping is an \ti{$L$-Lipschitz quotient mapping} if it is $L$-Lipschitz and $(1/L)$-co-Lipschitz.

One can use Lipschitz quotient mappings to lift rectifiable curves, as in the following lemma. This lemma was first proven in \cite[Lemma 4.4]{BJLPS99} and \cite[Lemma 2.2]{JLPS00}. Though stated there for $\RR^n$-targets, the proof works the same way in the setting below, as was observed in \cite[Lemma 4.3]{DK16} and \cite[Lemma 3.3]{GCDK16}, where one can find a proof written in this generality. Recall that a metric space is \ti{proper} if closed balls are compact.

\begin{lemma}\label{pathlifting}
Let $X$ be a proper metric space and $Y$ a metric space. Let $f \colon X\rightarrow Y$ be an Lipschitz quotient map with co-Lipschitz constant $c>0$, and let $\gamma \colon [0,T] \rightarrow Y$ be a $1$-Lipschitz curve with $\gamma(0)=f(x)$. Then there is a $(1/c)$-Lipschitz curve $\tilde{\gamma} \colon [0,T] \rightarrow X$ such that $\tilde{\gamma}(0)=x$ and $f \circ \tilde{\gamma} = \gamma$. 
\end{lemma}

\subsection{PI spaces and Lipschitz differentiability} \label{PILD}
In this subsection, we give a very brief introduction to PI spaces, as first defined in \cite{HK98}, and to Cheeger's theory of Lipschitz differentiation on such spaces \cite{Ch99}.

\subsubsection{PI spaces}
We now start to define PI spaces, the ambient spaces of Theorems \ref{tangentthm}, \ref{piecethm}, and \ref{Rnpiecethm}. As we will explain below, these are doubling metric measure spaces that support a Poincar\'e inequality. In this paper, we do not make serious use of the Poincar\'e inequality itself, using instead two of its important consequences: Proposition \ref{PIqc} and Cheeger's differentiation theory (subsection \ref{LDsec}). Nonetheless, we include this material as background. For a much more substantial introduction to these spaces, see \cites{HK98, HKST15, Ch99}.

We begin by defining the pointwise Lipschitz constant of a Lipschitz function.
\begin{definition}
Let $f \colon X\rightarrow\RR$ be a Lipschitz function. Then
$$ \Lip_f(x) = \limsup_{x\neq y\rightarrow x} \frac{|f(x)-f(y)|}{d(x,y)}.$$
\end{definition}

We now define the Poincar\'e inequality for metric measure spaces. Such a definition was first given by Heinonen and Koskela \cite{HK98}; the definition we state is an equivalent one due to work of Keith \cite{Ke03}.
\begin{definition}
Let $(X,d,\mu)$ be a metric measure space and $p\geq 1$. We say that $(X,d,\mu)$ \ti{admits a $p$-Poincar\'e inequality} (or simply a Poincar\'e inequality) if there exist constants $C,\lambda\geq 1$ such that, for every compactly supported Lipschitz function $f \colon X\rightarrow\RR$ and every open ball $B$ in $X$,
$$\fint_B |f-f_B|d\mu \leq C(\diam B)\left(\fint_{\lambda B} (\Lip_f)^p d\mu \right)^{1/p}.$$
Here the notations $\fint_E g d\mu$ and $g_E$ both denote the average value of the function $g$ on the set $E$, i.e., $\frac{1}{\mu(E)}\int_E g d\mu$.
\end{definition}

Finally, we define the PI spaces themselves.
\begin{definition}
A \textit{PI space} is a complete metric measure space $(X,d,\mu)$ such that $\mu$ is doubling and $(X,d,\mu)$ supports a Poincar\'e inequality.
\end{definition}

As mentioned above, we will not directly use the Poincar\'e inequality itself at any point in this paper, but rather some consequences. The following result of Semmes is one of the most useful properties of PI spaces.

\begin{prop}[Appendix A of \cite{Ch99}]\label{PIqc}
Let $(X,d,\mu)$ be a PI space. Then $X$ is quasiconvex; that is, there is a constant $K$ such that each pair of points $x,y\in X$ can be joined by a rectifiable curve of length at most $Kd(x,y)$.

The constant $K$ depends only on the constants associated to the doubling and Poincar\'e inequality properties of $X$.
\end{prop}

\subsubsection{Lipschitz differentiability}\label{LDsec}
In \cite{Ch99}, Cheeger proved that all PI spaces support a remarkable form of first-order calculus for Lipschitz functions. To explain this, we first define a class of spaces which support such a first-order calculus by fiat.

\begin{definition}[\cite{Ch99}, Theorem 17.1]\label{LDspace}
A metric measure space $(X,d,\mu)$ is called a \textit{Lipschitz differentiability space} if it satisfies the following condition. There are countably many Borel sets (``charts'') $U_i$ covering $X$, positive integers $n_i$ (the ``dimensions of the charts''), and Lipschitz maps $\phi_i\colon X\rightarrow\mathbb{R}^{n_i}$ with respect to which any Lipschitz function $f \colon X \rightarrow \R$ is differentiable almost everywhere, in the sense that for each $i$ and for $\mu$-almost every $x\in U_i$, there exists a unique $df(x) \in\mathbb{R}^{n_i}$ such that
\begin{equation} \label{LD}
\lim_{y\rightarrow x} \frac{|f(y) - f(x) -  df(x) \cdot(\phi_i(y)-\phi_i(x))|}{d(x,y)} = 0.
\end{equation}
Here $df(x) \cdot(\phi_i(y)-\phi_i(x))$ denotes the standard scalar product in $\mathbb{R}^{n_i}$.
\end{definition}
Note that the Borel measurability of the function $x\mapsto df(x)$ and the set of differentiability points of $f$ are consequences of this definition; see \cite[Remark 1.2]{BS13}.

The main result of \cite{Ch99} is that PI spaces are Lipschitz differentiability spaces.

\begin{theorem}[\cite{Ch99}]\label{cheeger}
Every PI space is a Lipschitz differentiability space. The dimensions $n_i$ of the charts are bounded above uniformly, depending only on the constants associated to the doubling condition and the Poincar\'e inequality.
\end{theorem}

In recent years, the study of Lipschitz differentiability spaces in their own right (independent of the stronger PI space assumptions) has become an active area of research, and we refer the reader to \cites{Ch99, Ke04, CK09, KM11, Ba15, GCD14, BL15, BL16, CKS15, EB16} for more background.

The following additional fact is quite simple but useful.
\begin{prop}
Let $(X,d,\mu)$ be a PI space and let $A\subset X$ be a closed subset of positive measure. Then $(A,d,\mu)$ is a Lipschitz differentiability space, and its charts can be taken to be the restrictions of charts on $X$.
\end{prop}

In the case of PI spaces, the previous Proposition is a simple consequence of Cheeger's Theorem \ref{cheeger}, the doubling property, and the Lebesgue density theorem. In fact, the analogous result holds for subsets of general Lipschitz differentiability spaces by Corollary 2.7 of \cite{BS13}.

\begin{remark}\label{assumptions}
Ahlfors regularity (or localized versions of it) are important in all our main results, Theorems \ref{tangentthm}, \ref{piecethm}, and \ref{Rnpiecethm}. However, the Poincar\'e inequality itself is not really used in the proofs. In Theorems \ref{tangentthm} and \ref{piecethm}, we use only the fact that $X$ is a Lipschitz differentiability space and that all tangents of $X$ are quasiconvex. This follows from the assumption that $X$ is a PI space, but it would also hold if we instead assumed that $X$ was just an RNP-differentiability space, in the sense of \cite{BL16} (see Theorem 6.6 in \cite{BL16}). RNP-differentiability is a strictly stronger assumption than Lipschitz differentiability (see \cite{Sc16}). However, by recent work of Eriksson-Bique \cite{EB16}, each RNP-differentiability space is essentially a union of subsets of PI spaces up to measure zero, and so we do not really lose any generality in our theorems with the stated assumptions.

In Theorem \ref{Rnpiecethm}, it would be enough to assume that $X$ was an Ahlfors regular Lipschitz differentiability space; see section \ref{Rnpiecesection}.
\end{remark}

\subsection{Carnot groups} \label{Carnotsection}
Carnot groups (equipped with their Carnot--Carath\'eodory metrics and Hausdorff measure) are metric measure spaces that naturally generalize Euclidean spaces from the perspective of many problems in geometry and analysis. In this subsection, we give some brief background on Carnot groups. For more, we refer the reader to \cite{Mo} or \cite[Chapter 2]{CDPT07}.

A Carnot group is a simply connected nilpotent Lie group $\G$ whose Lie algebra $\mathfrak{g}$ admits a stratification
$$\mathfrak{g} = V_1 \oplus \cdots \oplus V_s,$$
where the first layer $V_1$ generates the rest via $V_{i+1} = [V_1,V_i]$ for all $1 \leq i \leq s$, and we set $V_{s+1} = \{0\}$. The exponential map $\exp \colon \mathfrak{g} \rightarrow \G$ is a diffeomorphism, so choosing a basis for $\mathfrak{g}$ gives exponential coordinates for $\G$. For each $x \in \G$, we will use $L_x \colon \G \rightarrow \G$ to denote the left multiplication map $y \mapsto xy$.

A natural family of automorphisms of $\G$ are the dilations $\delta_{\lambda} \colon \G \rightarrow \G$, for $\lambda >0$. On the Lie algebra level, these are linear maps defined by 
\begin{equation}\label{dilation}
v \mapsto \lambda^i v, \hspace{0.3cm} \text{for } v \in V_i,
\end{equation}
and one can see that this gives a Lie algebra isomorphism. Conjugating this back to $\G$ by the exponential map defines $\delta_{\lambda}$. 

On any Carnot group $\G$, there are metrics that interact nicely with the translations and dilations, in the sense that each $L_x$ is an isometry and $\delta_{\lambda}$ scales distances by the factor $\lambda$.

\begin{definition}\label{homogdistance}
A metric $d \colon \G \times \G\rightarrow [0,\infty)$ is called a \textit{homogeneous distance} if it induces the manifold topology of $\G$, it is left-invariant, i.e.
$$ d(xy,xz) = d(y,z) \text{ for all } x,y,z \in \G,$$
and it is $1$-homogeneous with respect to the dilations $\delta_\lambda$ defined above:
$$ d(\delta_\lambda(x), \delta_\lambda(y)) = \lambda d(x,y) \text{ for all } \lambda>0 \text{ and } x,y\in\G.$$
\end{definition}

For example, given an inner product $\langle\cdot,\cdot\rangle$ on the horizontal layer $V_1$, the associated sub-Riemannian Carnot--Carath\'eodory metric $d_{cc}$ is an example of a homogeneous distance on $\G$. This is defined by
$$ d_{cc}(x,y) = \inf\{ \int_0^1 \langle \gamma'(t), \gamma'(t)\rangle ^{1/2} dt : \gamma \text{ horizontal curving joining } x \text{ to } y \},$$
where an absolutely continuous curve $\gamma \colon [0,1]\rightarrow\G$ is called \ti{horizontal} if $\gamma'(t) \in V_1$ for a.e. $t\in[0,1]$.

All homogeneous distances on a given Carnot group are bi-Lipschitz equivalent. For the purposes of this paper, a ``sub-Riemannian Carnot group'' means a Carnot group $\G$ equipped with such an inner product on $V_1$ and hence a Carnot--Carath\'eodory distance $d_{cc}$. As we will not consider any other distances on Carnot groups, we will often simply write ``Carnot group'' when we mean ``sub-Riemannian Carnot group''. In that case, the homogeneity of the group implies that $ (\G, d_{cc}, \mathcal{H}^Q)$ is an Ahlfors $Q$-regular metric measure space, where $Q = \Hdim(\G)$. 

If $\G$ is a Carnot group, there is a natural ``horizontal projection''
$$\pi \colon \G \rightarrow V_1 \simeq \R^n$$
obtained by composing $\exp^{-1}$ with the vector space projection $\mathfrak{g} \rightarrow V_1$. The following lemma summarizes the basic properties of $\pi$ that we will need below. These properties are standard, and are collected with proofs and/or references in \cite[Lemma 2.8]{DK16}.

\begin{lemma}\label{pifacts}
Let $\G$ be a Carnot group whose horizontal layer $V_1$ has dimension $n$, and let $\pi \colon \G \rightarrow V_1 \simeq \R^n$ be the associated horizontal projection.
\begin{enumerate}[\normalfont (i)]
\item\label{pihomomorphism} $\pi$ is a group homomorphism.
\item\label{pidilation} $\pi$ commutes with dilations: $\pi(\delta_{\lambda}(x)) = \lambda \pi(x)$ for all  $\lambda >0$.
\item\label{pitangent} For all $x\in \G$, every element of $\Tan(\G, x, \pi)$ is isometric to $(\G, 0, \pi)$. (See Section \ref{tangentsec} for the notation.)
\item\label{piLQ} $\pi$ is a Lipschitz quotient map onto $V_1 \simeq \R^n$.
\item\label{picurve} If $\gamma \colon [0,1]\rightarrow \G$ is a non-constant Lipschitz curve, then $\pi \circ \gamma$ is non-constant.
\item\label{pilift} If $\gamma \colon [0,1]\rightarrow V_1$ is a Lipschitz curve and $x \in \pi^{-1}(\gamma(0))$, then there is a \textit{unique} Lipschitz curve $\tilde{\gamma}\colon [0,1]\rightarrow \G$ such that $\pi(\tilde{\gamma}(t)) = \gamma(t)$ and $\tilde{\gamma}(0) = x$. 
\end{enumerate}
\end{lemma}

\subsection{Carnot subgroups}

Let $\G$ be a sub-Riemannian Carnot group with Lie algebra $\mathfrak{g}$ and horizontal layer $V_1 \subset \mathfrak{g}$, and let $\pi \colon \G \rightarrow V_1$ be the horizontal projection.

\begin{definition}
Let $V \subset V_1$ be a vector subspace, and let $\mathfrak{h} \subset \mathfrak{g}$ be the stratified Lie sub-algebra generated by $V$. The homogeneous subgroup $\H = \exp_{\G}(\mathfrak{h}) \subset \G$ is called the Carnot subgroup generated by $V$.
\end{definition}
The name ``Carnot subgroup'' was used for these objects in \cite{DK16}, where they played a similar role to that of vector subspaces in Euclidean space.

For $V$ fixed, $\H$ is the unique homogeneous subgroup of $\G$ with Lie algebra $\mathfrak{h}$. Indeed, any homogeneous subgroup of $\G$ is identified with its Lie algebra in exponential coordinates. Moreover, $\H$ is itself a Carnot group with horizontal layer $V$, and as a subset of $\G$ it is rectifiably connected with respect to the Carnot--Carath\'eodory metric.

Here we include a characterization of those subsets of $\G$ that coincide with Carnot subgroups. We will use it in the proof of Theorem \ref{blowup}, but it appears to be an interesting statement in itself. This argument is essentially that of Theorem 4.1 in \cite{DK16}.

\begin{prop} \label{Carnotsub}
Let $Y \subset \G$ be a closed, rectifiably connected subset with $0 \in Y$. If there is a vector subspace $V \subset V_1$ such that $\pi|_Y \colon Y \rightarrow V$ is a Lipschitz quotient map onto $V$, then $Y$ coincides with the Carnot subgroup generated by $V$.
\end{prop}

\begin{proof}
We first verify that $Y$ is a homogeneous subgroup of $\G$. 

Let $x,y \in Y$ and let $\gamma$ be a rectifiable curve in $Y$ from $0$ to $y$. Then $\pi \circ \gamma$ is a rectifiable curve in $V$ from $0$ to $\pi(y)$, and so its translate $\pi(x) + \pi \circ \gamma$ is a rectifiable curve in $V$ from $\pi(x)$ to $\pi(x)+\pi(y) = \pi(xy)$. As $\pi|_Y$ is a Lipschitz quotient map onto $V$, this curve has a lift to a rectifiable curve in $Y$ that begins at $x$, by Lemma \ref{pathlifting}. Note, though, that $L_{x} \circ \gamma$ is the unique lift of $\pi(x) + \pi \circ \gamma$ to $\G$; this follows immediately from the fact that $\pi$ is a group homomorphism. Thus, the curve $L_{x} \circ \gamma$ is contained in $Y$, so in particular $xy \in Y$. This shows that $Y$ is closed under multiplication.

Similar arguments show that $Y$ is closed under inversion and dilation. If $x \in Y$, then let $\gamma$ be a rectifiable curve in $Y$ from $0$ to $x$. Note that $L_{-x} \circ \gamma$ is also rectifiable in $\G$ and is the unique lift of $-\pi(x) + \pi \circ \gamma$ to $\G$ that begins at $0$. As $-\pi(x) + \circ \pi \circ \gamma$ is a rectifiable curve in $V$, the Lipschitz quotient property of $\pi|_Y$ ensures that this lift lies in $Y$. In particular, $x^{-1} \in Y$. Finally, if $\delta_{\lambda}$ is a dilation for $\G$, then $\delta_{\lambda} \circ \gamma$ is rectifiable in $\G$ with projection $\pi \circ \delta_{\lambda} \circ \gamma = \lambda \cdot \pi \circ \gamma$, which again lies in $V$. Once again, the Lipschitz quotient property of $\pi|_Y$ guarantees that $\delta_{\lambda} \circ \gamma$ lies in $Y$, so in particular $\delta_{\lambda}(x) \in Y$.

Thus, $Y$ is a closed, rectifiably connected, homogeneous subgroup of $\G$. Let $\mathfrak{h} \subset \mathfrak{g}$ be its Lie algebra so that $Y = \exp_{\G}(\mathfrak{h})$. We note that the horizontal component of $\mathfrak{h}$ is precisely $\pi(Y) = V$. Consequently, if $\H$ denotes the Carnot subgroup generated by $V$, then its Lie algebra is contained in $\mathfrak{h}$ and we automatically have $\H \subset Y$. For the reverse containment, we proceed as before: for $x \in Y$ and $\gamma$ a rectifiable curve in $Y$ from $0$ to $x$, the projection $\pi \circ \gamma$ is in $V$, which is the horizontal layer of $\H$. As $\H$ is a Carnot group itself, the unique lift of $\pi \circ \gamma$ to $\G$ that begins at $0$, namely $\gamma$, actually lies in $\H$. We conclude that $x \in \H$, which verifies that $Y = \H$.
\end{proof}

Let us remark that rectifiable connectivity of $Y$ is important in the above characterization. Indeed, for any subspace $V \subset V_1$, the full pre-image $\pi^{-1}(V)$ is a homogeneous subgroup of $\G$, and the restriction of $\pi$ to this subgroup is a Lipschitz quotient map onto $V$. In fact, the arguments given above make it clear that the Carnot subgroup generated by $V$ is precisely the subset
$$\H = \{x \in \pi^{-1}(V) : 0 \text{ and } x \text{ can be joined by a rectifiable path in } \pi^{-1}(V) \}.$$
As a basic example to illustrate this point, consider the first Heisenberg group $\H^1$ in exponential coordinates, so $\pi \colon \H^1 \rightarrow \R^2$ is just the linear projection onto the $xy$-plane. If $V$ is the $x$-axis, then $\pi^{-1}(V)$ is the homogeneous subgroup formed by the $xz$-plane. The Carnot subgroup generated by $V$ is, however, just the $x$-axis, which is precisely the set of points in the $xz$-plane that can be joined to $0$ by a finite-length curve lying in this plane.

\section{Tangents} \label{tangentsec}
In this section, we give definitions for the different types of tangent (or ``blowup'') constructions that we will need in the proofs. We also provide some useful facts on tangents of subsets at points of density, as well as some versions of a principle, usually attributed to Preiss, about moving basepoints in tangents.

\subsection{Tangents of metric spaces}
First, we recall the notion of Gromov--Hausdorff tangents. Suppose $(X,x)$ is a pointed metric space. A Gromov--Hausdorff tangent of $(X,x)$ at $x$ is any complete pointed metric space that is a pointed Gromov--Hausdorff limit of $(\lambda_i^{-1} X, x)$ for a sequence $\lambda_i \rightarrow 0$. Note that this is defined only up to isometry. We will introduce a precise (pseudo)-metric inducing this convergence below, when we add mappings to the picture. For additional background on pointed Gromov--Hausdorff convergence, see \cites{BBI01, DS97, GCD14}.

We denote the collection of all (isometry classes of) Gromov--Hausdorff tangents of $(X,x)$ by $\Tan(X,x)$. By standard compactness results for pointed Gromov--Hausdorff convergence (e.g., Theorem 8.1.10 of \cite{BBI01}), if $X$ is metrically doubling then $\Tan(X,x)$ is always non-empty. In fact, for any sequence $\lambda_i \rightarrow 0$, there is a subsequence for which $(\lambda_i^{-1} X, x)$ converges to an element of $\Tan(X,x)$.

\subsection{Tangents of subsets of Carnot groups}
If $E$ is a subset of a Carnot group $\G$ and $x\in E$, there is a related but distinct notion of a tangent of $E$ at $x$, namely the intrinsic tangents of $E$ in $\G$. These are closed subsets of $\G$ that are pointed Hausdorff limits of 
\begin{equation}\label{rescaling}
 \delta_{\lambda_i^{-1}}(x^{-1} E)
\end{equation}
for some sequence $\lambda_i \rightarrow 0$.

Recall that a collection of pointed sets $A_i$ in $\G$ converges to $A$ in the \textit{pointed Hausdorff sense} if and only if
$$ \lim_{i\rightarrow\infty} d_R(A_i,A) = 0\text{ for all } R>0,$$
where
$$ d_R(A,B) = \sup \{\dist(a,B) : a\in A\cap B(0,R)\} + \sup \{\dist(b,A) : b\in B\cap B(0,R)\}.$$
Observe that $d_R$ does not distinguish between a set and its closure, and hence we define intrinsic tangents to be closed sets.

We denote the collection of all these intrinsic tangents of $E$ at $x$ by $\Tan_{\G}(E,x)$. In contrast to $\Tan(E,x)$, distinct elements of $\Tan_{\G}(E,x)$ may be isometric as pointed metric spaces. By standard compactness results for pointed Hausdorff convergence (see, e.g., \cite[Lemma 8.2]{DS97}), if $E\subset\G$ and $x\in E$, then $\Tan_\G(E,x)$ is non-empty, and indeed any sequence of scales tending to zero yields a subsequence along which the sequence in \eqref{rescaling} converges in the pointed Hausdorff sense to a closed set.

Using this fact, one sees that if $\hat{E}\in\Tan_{\G}(E,x)$, then the pointed isometry class of the pointed metric space $(\hat{E},0)$ is an element of $\Tan(E,x)$. Conversely, if $(Y,y)\subset\Tan(E,x)$, then there is an element $\hat{E}\in\Tan_{\G}(E,x)$ and a pointed isometry from $(Y,y)$ to $(\hat{E},0)$. Moreover, the fact that $d_R$ does not distinguish between a set and its closure implies that $\Tan_{\G}(E,x) = \Tan_{\G}(\cl{E},x)$ for any $E \subset \G$ and $x \in E$.

\subsection{Tangents of metric spaces and mappings to Carnot groups}

Before introducing tangents of mappings, we need some additional definitions.

We will call a \textit{package} a triple of the form $(X,x,f \colon X\rightarrow \G)$, where $(X,x)$ is a pointed metric space, $\G$ is a sub-Riemannian Carnot group, and $f \colon X\rightarrow \G$ is a Lipschitz mapping. (This name is not standard; similar objects are called ``space-functions'' in \cites{Ke04,GCD14} and ``mapping packages'' in \cite{DS97}.)

Strictly speaking, a package is an equivalence class rather than a single triple; two packages $(X,x,f)$ and $(X',x',f')$ are considered equivalent if there is a surjective isometry $i \colon X\rightarrow X'$ such that $i(x)=x'$ and $f'\circ i = f$. Note that two equivalent packages have identical images in $\G$ (i.e., $f(X) = f'(X')$), and not merely isometric images.

Often, we will abuse notation and say that a package is ``complete'' or ``doubling'' if the underlying space is complete or doubling. On occasion, we will also use the phrase ``Lipschitz package" to emphasize that the mapping $f$ is Lipschitz, even though this is part of the definition.

We now explain a notion of distance on the collection of all such packages with a fixed Carnot target $\G$. This notion is a minor variation of that defined in \cite{GCD14} when $\G=\RR^n$, which is in turn based on standard metrizations of pointed Gromov--Hausdorff distance like those in \cites{LD11, BBI01, DS97}.

\begin{definition}\label{epsilonisomdef}
A map $\phi \colon (X,d,x)\rightarrow (Y,d',y)$ between pointed metric spaces is called an \textit{$\epsilon$-isometry} if 
\begin{enumerate}[\normalfont (i)]
\item \label{isom1} For all $a,b \in B_X(x,1/\epsilon)$, we have $|d'(\phi(a),\phi(b)) - d(a,b)| < \epsilon$, and
\item \label{isom2} for all $\epsilon \leq r\leq 1/\epsilon,$ we have $N_\epsilon(\phi(B_X(x,r))) \supseteq B_Y(y,r-\epsilon)$.
\end{enumerate}
Here $N_\epsilon(E)$ denotes the open $\epsilon$-neighborhood of a subset $E$ in a metric space. Note that we do not ask that $\phi(x)=y$, although it follows from the definition that $d'(\phi(x),y)\leq 2\epsilon$.
\end{definition}

\begin{definition}
If $(X,x,f\colon X\rightarrow \G)$ and $(Y,y,g\colon Y\rightarrow \G)$ are packages, we define
\begin{align*}
\tilde{D}((X,d,x,f), (Y,d',y,g)) = \inf \bigg\{&\epsilon>0: \text{there exist } \phi \colon (X,d,x)\rightarrow (Y,d',y) \text{ and }\\
&\psi \colon (Y,d',y)\rightarrow (X,d,x) \text{ that are }\epsilon\text{-isometries,}\\
&\text{for which }\sup_{B(x,1/\epsilon)} d_{cc}(f,g\circ\phi)<\epsilon \text{ and }\\
&\sup_{B(y,1/\epsilon)}d_{cc}(g,f\circ\psi)<\epsilon\bigg\}.
\end{align*}
\end{definition}

We remark that any package $(X,x,f \colon X\rightarrow \G)$ is at zero $\tilde{D}$-distance from a complete package $(\overline{X},x,f)$, where $\overline{X}$ is the metric completion of $X$ and $f$ is identified with its unique completion to $\overline{X}$.

The following lemma is essentially a duplicate of Lemma 2.3 in \cite{GCD14}. Although that lemma was stated only for $\G=\RR^n$, the proof is identical and so we omit it.
\begin{lemma}\label{Dproperties}
If we define $D = \min\{\tilde{D}, 1/2\}$, then $D$ is a ``pseudo-quasi-metric", by which we mean the following.

\begin{enumerate}[\normalfont (i)]
\item\label{sym} $D$ is finite, non-negative, and symmetric.
\item\label{pseudometric} The $D$-distance between two doubling packages $(X,x,f)$ and $(Y,y,g)$ is zero if and only if there is a surjective isometry $i\colon \obar{X}\rightarrow \obar{Y}$ such that $g\circ i = f$, where $g$ and $f$ are identified with their extensions to the completions $\obar{X}$ and $\obar{Y}$.
\item\label{quasimetric} $D$ satisfies the quasi-triangle inequality
$$ D\left((X,x,f), (Z,z,h)\right) \leq 2\left(D\left((X,x,f),(Y,y,g)\right) + D\left((Y,y,g),(Z,z,h)\right)\right).$$
\end{enumerate}
\end{lemma}

Although the function $D$ is not a metric, we will still say that a sequence of packages 
$$(X_n, x_n,f_n \colon X_n\rightarrow \G)$$
``converges in $D$'' to a package $(X,x,f \colon X\rightarrow \G)$ if
$$D((X_n, x_n,f_n), (X,x,f))\rightarrow 0 \text{ as } n\rightarrow\infty.$$
The convergence in $D$ of a sequence of packages implies that the pointed metric spaces converge in the pointed Gromov--Hausdorff sense. Conversely, if $(X_n, x_n, f_n)$ are $C$-doubling, $L$-Lipschitz packages mapping to a fixed Carnot group $\G$, the sequence $\{f_n(x_n)\}$ is bounded, and $(X_n, d_n, x_n)\rightarrow (X,d,x)$ in the pointed Gromov--Hausdorff sense, then there is a subsequence  $(X_{n_k}, x_{n_k}, f_{n_k})$ and a Lipschitz function $f\colon X\rightarrow G$ such that
$$ (X_{n_k}, x_{n_k}, f_{n_k}) \rightarrow (X, x, f) $$
in the metric $D$. This follows by a standard Arzel\`{a}-Ascoli type argument.

In particular, if $(X_n, x_n, f_n)$ is a sequence of $C$-doubling, $L$-Lipschitz packages mapping to a fixed Carnot group $\G$, and $\{f_n(x_n)\}$ are bounded, then there is a subsequence that converges in $D$.

\begin{lemma}\label{GHproperties}
The following properties are preserved under convergence of a sequence of packages $(X_i,x_i,f_i \colon X_i\rightarrow \G) \rightarrow (X,x,f \colon X\rightarrow \G)$:
\begin{enumerate}[\normalfont (i)]
\item If the functions $f_i$ are all $L$-Lipschitz, then so is $f$.
\item If the functions $f_i$ are all $L$-bi-Lipschitz, then so is $f$.
\item If the spaces $X_i$ are uniformly doubling metric spaces, then $X$ is doubling.
\item If the spaces $X_i$ are uniformly quasi-convex, then $X$ is quasi-convex.
\item If the spaces $X_i$ are uniformly Ahlfors $Q$-regular, then $X$ is Ahlfors $Q$-regular.
\end{enumerate}
\end{lemma}

\begin{proof}
The first four of these properties are easy to check, and the fifth can be found in, e.g., Lemma 8.29 of \cite{DS97}.
\end{proof}

Now suppose that $(X,x)$ is a pointed metric space, $\G$ is a Carnot group, and $f \colon X\rightarrow \G$ is a Lipschitz map. A tangent of $(X,x,f)$ at $x$ is any limit
\begin{equation} \label{tangentOfPackage}
(\hat{X},\hat{x},\hat{f}) = \lim_{i\rightarrow\infty} \left(\frac{1}{\lambda_i}X, x, \delta_{\frac{1}{\lambda_i}}(f(x)^{-1}\cdot f(\cdot))\right),
\end{equation}
where $\lambda_i$ is a sequence of positive real numbers tending to $0$.

We denote the collection of all complete tangents of $(X,x,f)$ by $\Tan(X,x,f)$. We note again that an element of $\Tan(X,x,f)$ is defined only up to isometries of $\hat{X}$ that preserve $\hat{f}$, but that the image of a tangent in $\G$ is fixed.

If $(X,x,f \colon X\rightarrow \G)$ is a package as defined above and $X$ is doubling, then $\Tan(X,x,f)$ is always non-empty, by the above-mentioned standard facts about Gromov--Hausdorff compactness.

Observe also that if $(\hat{X},\hat{x},\hat{f})\in\Tan(X,x,f)$, then $\hat{f}(\hat{X})\subset E$ for some $E\in\Tan_{\G}(f(X),f(x))$. Indeed, if $(\hat{X},\hat{x},\hat{f})$ is obtained by a sequence $\lambda_i \rightarrow 0$ as in \eqref{tangentOfPackage}, then there is a subsequence of $\lambda_i$ along which $(\delta_{\lambda_i^{-1}}(f(x)^{-1}f(X))$ converges to an element of $\Tan_{\G}(f(X),f(x))$.

\subsection{Tangents of measures on Carnot groups}
We will also need to define the notion of tangent measures on Carnot groups, following \cite{Mat05} (which was in turn inspired by \cite{Pr87} in Euclidean space). Recalling that $L_{x}$ denotes left translation by $x \in \G$ and $\delta_{\lambda}$ dilation by $\lambda >0$, define an affine map on $\G$ by
$$ T_{x,\lambda} = \delta_{\lambda^{-1}} \circ L_{x^{-1}}.$$
If $\nu$ is a Radon measure on $\G$, we say that $\hat{\nu}$ is a \textit{tangent measure} of $\nu$ at $x \in \G$ if there are positive sequences $c_i$ and $\lambda_i$, with $\lambda_i \rightarrow 0$, such that
$$ c_i (T_{x,\lambda_i})_{\#} \nu \rightarrow \hat{\nu} \text{ weakly as } i\rightarrow\infty.$$
We denote the collection of tangent measures to $\nu$ at $x$ by $\Tan(\nu,x)$.

The existence of tangent measures in our setting will be guaranteed by the following lemma. The version in \cite[Theorem 14.3]{Mat95} is stated for measures on Euclidean space, but the proof is identical.

\begin{lemma}\label{tangentmeasureexists}
Let $\nu$ be a Radon measure on a Carnot group $\G$. If $x \in\G$ and
$$ \limsup_{r \rightarrow 0} \frac{\nu(B(x,2r))}{\nu(B(x,r))} < \infty,$$
then every sequence $\lambda_i \rightarrow 0$ contains a subsequence $\lambda_{i_j}$ such that the measures
$$ \nu(B(x,\lambda_{i_j}))^{-1} (T_{x,\lambda_{i_j}})_{\#} \nu$$
converge to a tangent measure of $\nu$ at $x$.
\end{lemma}

For any element of $\Tan(X,x)$, $\Tan(X,x,f)$, $\Tan_{\G}(E,x)$, or $\Tan(\nu,x)$, we say that the tangent object is \textit{subordinate} to the sequence of scales $\lambda_i \rightarrow 0$ along which the limiting sequence is taken. 

We should remark that, for a tangent measure $\hat{\nu} \in \Tan(\nu,x)$ subordinate to $\lambda_i$, it need not be true that the normalization constants $c_i$ are equal to $\nu(B(x,\lambda_i))^{-1}$. The following lemma, however, indicates that this is nearly true. As above, the proof in Remarks 14.4 of \cite{Mat95} for Euclidean space also works in this setting.

\begin{lemma} \label{tanmeasures}
Let $\nu$ be a Radon measure on a Carnot group $\G$, and suppose that $x \in \G$ has
$$\limsup_{r\rightarrow 0} \frac{\nu(B(x,2r))}{\nu(B(x,r))} < \infty.$$
If $\hat{\nu} \in \Tan(\nu,x)$ is subordinate to the sequence $\lambda_i$, then there is a subsequence $\lambda_{i_j}$ for which
$$\hat{\nu} = c \cdot \lim_{j \rightarrow \infty} \nu(B(x,\lambda_{i_j}))^{-1} (T_{x,\lambda_{i_j}})_{\#} \nu$$
with $c>0$ a constant.
\end{lemma}

The following lemma will also be useful during our discussion of tangent measures. Once gain, the proof in Lemma 14.7 of \cite{Mat95} works in the present setting as well.

\begin{lemma}\label{tanmeasureAR}
Let $\nu$ be a Radon measure on a Carnot group $\G$. Suppose that $x \in \G$ has
\begin{equation}\label{tanmeasureAReqn}
0 < \liminf_{r\rightarrow 0} \frac{\nu(B(x,r))}{r^Q} \leq \limsup_{r\rightarrow 0} \frac{\nu(B(x,r))}{r^Q} <\infty,
\end{equation}
with $t\in (0,1]$ the ratio of the second quantity in \eqref{tanmeasureAReqn} to the third.

Then for every $\hat{\nu} \in \Tan(\nu,x)$, there is a positive number $c>0$ such that 
$$ tcr^Q \leq \hat{\nu}(B(z,r)) \leq cr^Q $$
for all $z\in \supp(\hat{\nu})$ and all $r>0$.
\end{lemma}

\subsection{Tangents at points of density}

In this subsection, we collect some facts about how tangents of spaces, mappings, and measures behave when taken at density points of subsets. The main principle, expressed in a few different forms, is that tangents of subsets at points of density agree with tangents of the ambient space.

We start with the principle for tangents of measures in Carnot groups. Recall that if $\nu$ is a Radon measure on a Carnot group $\G$ and $S$ is a measurable subset of $\G$, then $x$ is a point of $\nu$-density of $S$ if
$$ \lim_{r\rightarrow 0} \frac{\nu(B(x,r)\setminus S)}{\nu(B(x,r))} = 0.$$
If $\nu$ is pointwise doubling (see subsection \ref{basicdefs}), then $\nu$-a.e. point is a point of $\nu$-density. (See, e.g., \cite[Section 3.4]{HKST15}.)

\begin{lemma}\label{subsettangentmeasure}
Let $\G$ be a Carnot group, $\nu$ be a Radon measure on $\G$, and $S$ a $\nu$-measurable subset. If $x$ is a point of $\nu$-density for $S$, then $\Tan(\nu|_S, x) = \Tan(\nu, x)$, and any $\hat{\nu}$ in both of these collections is subordinate to the same sequence of scales in both collections.
\end{lemma}

\begin{proof}
The corresponding statement for measures in Euclidean space appears in Lemma 14.5 of \cite{Mat95}, and the proof in the Carnot group setting is identical.
\end{proof}

Next we wish to prove a version of this principle for elements of $\Tan_{\G}$.

\begin{lemma}\label{subsettangent}
Let $\G$ be a Carnot group, and let $Y$ be a compact subset of $\G$. Let $\nu$ be a Radon measure on $\G$ such that there are constants $Q>0$ and $0<\alpha<\beta$ with
\begin{equation}\label{Ybound}
\alpha r^Q \leq \nu(B(y,r)) \leq \beta r^Q
\end{equation}
for all $y\in Y$ and $0< r < r_0$.

If $F$ is a measurable subset of $Y$ and $y \in F$ is a point of $\nu$-density for $F$, then $\Tan_{\G}(F,y) = \Tan_{\G}(Y,y)$.
\end{lemma}

\begin{proof}
As $F\subseteq Y$, it suffices to show that 
$$ \limsup_{r \rightarrow 0} \frac{\dist(Y\cap B(y,r),F)}{r} = 0.$$
Fix $r<r_0$. Let $d=\sup_{z \in Y \cap B(y,r)} \dist(z,F)$. Then there is a point $z \in Y\cap B(y,r)$ with 
$$ B(z,d/2) \subseteq B(y,2r)\setminus F$$
and hence
$$ \frac{\nu(B(y,2r)\setminus F)}{\nu(B(y,2r)} \geq \frac{\alpha (d/2)^Q}{\beta (2r)^Q} = \frac{\alpha}{4^Q \beta} \left(\frac{d}{r}\right)^Q.$$
Since $y$ is a point of $\nu$-density of $F$, it follows that $d/r$ tends to $0$ with $r$, as desired.
\end{proof}

Finally, we state the version for packages.

\begin{lemma}\label{densitytangent}
Let $X$ be a complete, Ahlfors $Q$-regular metric space, $A \subset X$ a subset, $\G$ a Carnot group, and $(A,x,f \colon A \rightarrow \G)$ a package. Let $E\subset A$ be a subset of $A$, and suppose that $x \in E$ satisfies
\begin{equation}\label{outerdensity}
\lim_{r\rightarrow 0} \frac{\mathcal{H}^Q(B(x,r) \setminus E)}{\mathcal{H}^Q(B(x,r))} = 0.
\end{equation}
Then $\Tan(E,x,f) = \Tan(A,x,f),$ and any element in both of these collections is subordinate to the same sequence of scales in both collections.
\end{lemma}

\begin{proof}
The proof is essentially the same as the proof of, e.g., Proposition 3.1 in \cite{LD11}. The main point is that the collections of Gromov--Hausdorff tangents $\Tan(E,x)$ and $\Tan(A,x)$ coincide. In fact,
$$\Tan(E,x) = \Tan(X,x) = \Tan(A,x),$$
and any element in these collections is subordinate to the same sequence of scales in all collections.
\end{proof}

\begin{remark}
For $A \subset X$ as in the previous lemma and $x \in A$ a point of $\mathcal{H}^Q$-density for $A$, the domain of any fixed tangent map $\hat{f}$ to $f$ at $x$ is an element of $\Tan(A,x) = \Tan(X,x)$. Thus, it makes sense (though perhaps formally it is an abuse of notation) to write such things as $(\hat{X},\hat{x},\hat{f}) \in \Tan(A,x,f)$. By this, we do not mean to suggest that $f$ is defined on all of $X$.

This is done, for example, in the statement of Theorem \ref{tangentthm}, and it will be done throughout Section \ref{LQsec}. Whenever this is written, though, we implicitly mean that $x \in A$ is a point of $\mathcal{H}^Q$-density for $A$, so that $\Tan(A,x) = \Tan(X,x)$.
\end{remark}

\subsection{Moving basepoints and tangents of tangents}

In this subsection, we present some versions of Preiss's principles (see \cite{Pr87}) that (a) tangent objects with moved basepoints are still tangent objects, and (b) tangents of tangents are tangents. Since we have defined many types of tangent objects above, one may imagine many versions of these principles. We only state the versions that we use in the proofs of the main theorems below.

The first is a ``moving basepoint'' theorem for tangent measures. The version we state, which applies to Carnot groups (and more general metric groups) is due to Mattila \cite[Proposition 2.15]{Mat05}. Recall from section \ref{basicdefs} the definition of a pointwise doubling measure.

\begin{prop}\label{mattilamove}
Let $\nu$ be a pointwise doubling Radon measure on a Carnot group $\G$. Then for $\nu$-a.e. $x \in \G$ and all $\hat{\nu} \in \Tan(\nu,x)$, we have that
$$ (T_{y,\lambda})_{\#} \hat{\nu} \in \Tan(\nu,x)$$
for all $\lambda > 0$ and all $y \in \supp(\hat{\nu})$.
\end{prop}

In addition, we need the following ``tangents of tangents'' type statement, which is the final ingredient for proving Theorem \ref{tangentthm}.

\begin{prop}\label{tangentsoftangents}
Let $(X,d,\mu)$ be a doubling metric measure space and $\G$ a Carnot group. Let $A \subset X$ be a compact subset, and let $f \colon A\rightarrow \G$ be Lipschitz.

Then for $\mu$-a.e. $x \in A$, the following holds. If $(Y,y,g) \in \Tan(A,x,f)$ is a tangent, $y' \in Y$ is a point,
and $(Z,z,h) \in \Tan(Y,y',g)$ is bi-Lipschitz, then there is a tangent $(W,w,j) \in \Tan(A,x,f)$ that is bi-Lipschitz.
\end{prop}

We defer the proof of Proposition \ref{tangentsoftangents} to the Appendix.

\section{Lipschitz quotient and bi-Lipschitz maps as tangents} \label{LQsec}

In this section, we prove the main result concerning tangents of Lipschitz maps from PI spaces into Carnot groups, which is Theorem \ref{blowup} below. This contains the statement of Theorem \ref{tangentthm}, as well as further properties of tangents that will be needed in the following sections to prove Theorem \ref{piecethm}. Here we will rely heavily on the notion of Lipschitz quotient maps introduced in Definition \ref{LQdef}.

Lipschitz quotient maps enter the theory of Lipschitz differentiability spaces through the following result. It was proven in Theorem 5.56 of \cites{Sc13, Sc16} (see equation (5.96) therein) and in Corollary 5.1 of \cite{GCD14}. A stronger version can be found in Theorem 1.11 of \cite{CKS15}, but the version below suffices for our purposes.

\begin{prop}\label{LQblowup}
Let $(X,d,\mu)$ be a doubling Lipschitz differentiability space with a chart $(U,\phi\colon X\rightarrow \RR^k)$. Then for almost every $x\in X$, there is a constant $L\geq 1$ such that for every $(\hat{X},\hat{x},\hat{\phi})\in \Tan(X,x,\phi)$, the mapping $\hat{\phi}$ is a Lipschitz quotient map of $\hat{X}$ onto $\RR^k$ with constant $L$.
\end{prop}

We can now state the main result of this section, which, as mentioned above, contains Theorem \ref{tangentthm}. Many of the ideas in the proof are similar to those found in the proof of Theorem 4.1 in \cite{DK16}.

\begin{theorem} \label{blowup}
Let $X$ be an Ahlfors $Q$-regular PI space, $A \subset X$ a subset, and $\G$ a sub-Riemannian Carnot group. Suppose $f \colon A \rightarrow \G$ is a Lipschitz mapping such that $\mathcal{H}^Q(f(A)) > 0$.
\begin{enumerate}[\normalfont (i)]
\item\label{blowupi} For $\mathcal{H}^Q$-almost every $x \in A$, there is a Carnot subgroup $\H_x \subset \G$ for which every $(\hat{X}, \hat{x}, \hat{f}) \in \Tan(A,x,f)$ is a Lipschitz quotient map onto $\H_x$.

\item\label{blowupii} There is a subset $E \subset A$, with $\mathcal{H}^Q(f(A \setminus E)) = 0$, on which $\Hdim(\H_x) = Q$, and at each $x \in E$, there exists $(\hat{X}, \hat{x}, \hat{f}) \in \Tan(A,x,f)$ such that $\hat{f} \colon \hat{X} \rightarrow \H_x$ is bi-Lipschitz.
\end{enumerate}
\end{theorem}

The remainder of this section is devoted to the proof of Theorem \ref{blowup}. Thus, we fix an Ahlfors $Q$-regular PI space $X$, a subset $A\subset X$, a Carnot group $\G$, and a Lipschitz function $f \colon A\rightarrow \G$. 

We can assume, for the sake of convenience, that $A$ is compact and is contained in a single $k$-dimensional differentiability chart $U$ of $X$. Indeed, if one replaces $A$ by its closure $\overline{A}$ and extends $f$ to $\overline{A}$, the assumptions continue to hold and the conclusions for general $A$ follow. We can then decompose a closed set $A$ up to measure zero into countably many compact subsets of charts. Here we use Lemma \ref{densitytangent} to ensure that, when passing to subsets, $\Tan(A,x,f)$ remains the same, except possibly on a measure-zero set of points $x$.

Let $V_1 = \R^n$ be the horizontal layer of $\G$, and let $\pi \colon \G \rightarrow \R^n$ be the global chart for $\G$. Then $\pi \circ f \colon A \rightarrow \R^n$ is Lipschitz, so for almost every $x \in A$, there is a unique linear map $Df_x \colon \R^k \rightarrow \R^n$ with
\begin{equation}\label{Dfx}
\pi \circ f(y) - \pi \circ f(x) = Df_x(\phi(y)-\phi(x)) + o(d(x,y))
\end{equation}
for $y \in X$. For such $x \in A$, let $\H_x$ be the Carnot subgroup of $\G$ generated by the subspace $Df_x(\R^k) \subset \R^n$.

\begin{lemma}\label{blowuponto}
For almost every $x\in A$ and every $(\hat{X},\hat{x},\hat{f}) \in \Tan(A,x,f)$, we have $\hat{f}(\hat{X}) = \H_x$.
\end{lemma}

\begin{proof}
Let $x\in A$ be a point of density of $A$ in $X$ as well as a point of differentiability of $f$. This defines a Carnot subgroup $\H_x\subseteq\G$ as above.

Consider any $(\hat{X},\hat{x},\hat{f})\in\Tan(A,x,f)$. By passing to a subsequence, we may also consider 
$$(\hat{X},\hat{x},\hat{\phi}\colon\hat{X}\rightarrow\RR^k)\in\Tan(A,x,\phi)$$
subordinate to the same sequence of scales.

The defining property \eqref{Dfx} of $Df_x$ ensures that
$$\pi \circ \hat{f} = Df_x \circ \hat{\phi}.$$
For ease, let $V = Df_x(\R^k)$. Observe that $\hat{\phi}$ and $Df_x$ are Lipschitz quotients onto their images, the former by Proposition \ref{LQblowup} and the latter by linearity. It follows that $\pi \circ \hat{f} \colon \hat{X} \rightarrow V$ is a Lipschitz quotient. We first claim that $\pi|_{\hat{f}(\hat{X})} \colon \hat{f}(\hat{X}) \rightarrow V$ is also a Lipschitz quotient.

To verify this, consider any ball $B(\hat{f}(y),r)$ in $\hat{f}(\hat{X})$. If $L$ is the Lipschitz constant of $f$, then the ball $B(y,r/L)$ in $\hat{X}$ has
$$\hat{f}(B(y,r/L)) \subset B(\hat{f}(y),r).$$
Moreover, if we let $M$ be the co-Lipschitz constant of $\pi \circ \hat{f}$, then we find that
$$B_V(\pi (\hat{f}(y)), r/LM) \subset \pi \circ \hat{f}( B(y,r/L) ) \subset \pi( B(\hat{f}(y),r) ).$$
As $\pi$ is also Lipschitz, this shows that $\pi|_{\hat{f}(\hat{X})}$ is a Lipschitz quotient mapping.

We also observe that $\hat{f}(\hat{X})$ is rectifiably connected. For this, first note that $\hat{X}\in\Tan(A,x)=\Tan(X,x)$ by Lemma \ref{densitytangent}, and therefore $\hat{X}$ is quasiconvex by Proposition \ref{PIqc} and Lemma \ref{GHproperties}. It follows that $\hat{f}(\hat{X})$ is rectifiably connected, as it is the Lipschitz image of $\hat{X}$. The lemma now follows from Proposition \ref{Carnotsub}.
\end{proof}

The following Lemma will establish Theorem \ref{blowup}\eqref{blowupi}.

\begin{lemma}\label{blowupLQ}
For almost every $x\in A$ and every $(\hat{X},\hat{x},\hat{f}) \in \Tan(A,x,f)$, the tangent map $\hat{f} \colon \hat{X} \rightarrow \H_x$ is a Lipschitz quotient.
\end{lemma}

\begin{proof}
Let $x$ be a point at which the conclusion of Lemma \ref{blowuponto} holds, and fix $(\hat{X},\hat{x},\hat{f}) \in \Tan(A,x,f)$. By Lemma \ref{blowuponto}, $\hat{f}(\hat{X})=\H_x$.

Let $y \in \hat{X}$ and $v \in \H_x$ be arbitrary. In order to show that $\hat{f}\colon \hat{X}\rightarrow\H_x$ is a Lipschitz quotient, it suffices to show that there is $u \in \hat{X}$ with $\hat{f}(u)=v$ and $\hat{d}(y,u) \leq C \cdot d_{cc}(\hat{f}(y),v)$, with $C$ a uniform constant.

To this end, let us note that $\H_x$ is quasiconvex, so there is a curve $\gamma$ in $\H_x$ from $\hat{f}(y)$ to $v$ with length at most $C_1 \cdot d_{cc}(\hat{f}(y),v)$. Its projection $\pi \circ \gamma$ then has length at most $C_2 \cdot d_{cc}(\hat{f}(y),v)$. From the argument in the previous lemma, we know that $\pi \circ \hat{f}$ is a Lipschitz quotient map onto $V = \pi(\H_x)$, so by Lemma \ref{pathlifting} the curve $\pi \circ \gamma$ has a lift $\alpha$ to $\hat{X}$ that begins at $y$ and has length at most $C_3 \cdot d_{cc}(\hat{f}(y),v)$.

Observe that $\hat{f} \circ \alpha$ is a lift of $ \pi \circ \gamma$ through $\pi$ that begins at $\hat{f}(y)$. Uniqueness of lifts through $\pi$ (Lemma \ref{pifacts}) ensures that $\hat{f} \circ \alpha$ coincides with $\gamma$. In particular, if $u \in \hat{X}$ is the endpoint of $\alpha$, then $\hat{f}(u) = v$. Moreover, we see that
$$\hat{d}(y, u) \leq \length(\alpha) \leq C_3 \cdot d_{cc}(\hat{f}(y),v),$$
as desired.
\end{proof}

To establish the second part of Theorem \ref{blowup}, we will first need some results of David and Semmes. 

The first is a special case of a result of Semmes \cite[Theorem 10.1]{Se00}, which is in turn based on earlier work of David \cite{Da88} and Jones \cite{Jo88}. The full Theorem 10.1 of Semmes is much more general than we require; for an explanation of how our version follows from the general results of \cite{Se00} and \cite{Da88}, see \cite[Section 6.1]{GCD14}.

\begin{theorem}\label{bigpieces}
Let $M$ and $N$ be complete, Ahlfors $Q$-regular metric spaces, and let $g \colon M \rightarrow N$ be a Lipschitz quotient map from $M$ onto $N$.

Then for each ball $B\subset M$, the map $g$ is $\beta$-bi-Lipschitz on a subset of $B$ with $\HH^Q$-measure at least $\alpha\HH^Q(B)$. The constants $\alpha,\beta>0$ depend only on the Lipschitz quotient constants of $g$ and the Ahlfors regularity constants of $N$.
\end{theorem}

In fact, all we will need from Theorem \ref{bigpieces} is that the mapping in question is bi-Lipschitz on a subset of positive measure, which one could also view as a special case of Theorem \ref{BLthm} below.

For the second result of David and Semmes that we use, recall the notion of David--Semmes regularity from Definition \ref{DSregulardef}.

\begin{lemma} \label{DSregular}
Let $M$ be a complete, Ahlfors $Q$-regular metric space, let $N$ be a complete and doubling metric space, and let $A \subset M$ be a compact subset. Suppose $f \colon A \rightarrow N$ is a Lipschitz map with $\mathcal{H}^Q(f(A)) >0$. Then there is a subset $E \subset A$ with $\mathcal{H}^Q(f(A \setminus E)) = 0$ such that, at each $x \in E$, there is $(\hat{M},\hat{x},\hat{f}) \in \Tan(A,x,f)$ for which $\hat{f}$ is David--Semmes regular.
\end{lemma}

\begin{proof}
This lemma is essentially a restatement of Proposition 12.8 in \cite{DS97}, and the proof there establishes the version we have stated. We simply make two remarks about our slightly different statement.

First, the notion of tangent in \cite{DS97} is a metric notion which takes Gromov--Hausdorff tangents of both the domain and range, whereas our object $\Tan(A,x,f)$ uses the intrinsic scaling of $\G$ in the target. However, the two notions are isometrically identified, as noted in Section \ref{tangentsec}.

Second, \cite[Proposition 12.8]{DS97} concludes only that there exists a ``weak tangent" mapping that is David--Semmes regular. However, an inspection of the proof shows that at all points of density for the set $A \setminus f^{-1}(E_{\rho, \lambda})$ in $M$, for $\rho$ small enough and $\lambda$ large enough, there is a tangent map that is David--Semmes regular. Our stated conclusion follows from the fact that the measure of $E_{\rho, \lambda} \subset N$ tends to 0 as $\rho \rightarrow 0$ and $\lambda \rightarrow \infty$.
\end{proof}

For our fixed mapping $f \colon A\rightarrow\G$ in this section, let $E \subset A$ be the subset given by Lemma \ref{DSregular}. 

\begin{lemma}\label{blowupbilip}
For almost every $x \in E$, there is $(\hat{X},\hat{x},\hat{f}) \in \Tan(A,x,f)$ such that $\hat{f} \colon \hat{X} \rightarrow \H_x$ is bi-Lipschitz.
\end{lemma}

\begin{proof}
Fix $x \in E$ such that the conclusions of Lemma \ref{blowupLQ} and Proposition \ref{tangentsoftangents} hold. By Lemma \ref{DSregular} there is $(Y,y,g) \in \Tan(A,x,f)$ that is David--Semmes regular. Moreover, the veracity of Lemma \ref{blowupLQ} at $x$ guarantees that $g$ is a Lipschitz quotient map of $Y$ onto the Carnot subgroup $\H_x$.

As $g$ is David--Semmes regular and $Y$ is Ahlfors $Q$-regular, Lemma 12.3 of \cite{DS97} implies that $g(Y)= \H_x$ has Hausdorff dimension $Q$. Since $\H_x$ is a Carnot group, it is therefore also Ahlfors $Q$-regular. Thus, $g$ is a Lipschitz quotient map between Ahlfors $Q$-regular spaces. By Theorem \ref{bigpieces}, there is a subset $Z \subset Y$ of positive measure for which $g|_Z$ is bi-Lipschitz. 

Now let $z \in Z$ be a point of $\HH^Q$-density of $Z$ in $Y$, and consider any element $(\hat{Z},\hat{z},\hat{g})\in\Tan(Z,z,g) = \Tan(Y,z,g)$. As $g|_Z$ is bi-Lipschitz, so is $\hat{g}$. Now, by Proposition \ref{tangentsoftangents}, there is a tangent $(W,w,h) \in \Tan(A,x,f)$ for which $h$ is bi-Lipschitz. Note also that $h(\hat{Z}) = \H_x$ by Lemma \ref{blowuponto}. This completes the proof of Theorem \ref{tangentthm}.
\end{proof}

\begin{proof}[Proof of Theorem \ref{blowup} (and hence Theorem \ref{tangentthm})]
Part \eqref{blowupi} is contained in Lemma \ref{blowupLQ}, while part \eqref{blowupii} is contained in Lemma \ref{blowupbilip} and Lemma \ref{DSregular}.
\end{proof}

\subsection{Carnot unrectifiability}\label{carnotunrectsection}
Here we briefly record an immediate consequence of Theorem \ref{tangentthm} that concerns a strong notion of unrectifiability in the Carnot setting. 

In \cite{AK00}, Ambrosio and Kirchheim introduced the notion of strong unrectifiability: A metric space $X$ is \textit{strongly $Q$-unrectifiable} if $\HH^Q(f(X))=0$ for every $N\in\mathbb{N}$ and every Lipschitz map $f \colon X\rightarrow\RR^N$. This is a stronger notion than the classical pure unrectifiability, since no subset of Euclidean space with positive $\HH^Q$-measure can be purely $Q$-unrectifiable. Nonetheless, Ambrosio and Kirchheim exhibited non-trivial strongly $Q$-unrectifiable spaces for all $Q>0$.

In \cite[Theorem 1.4]{GCD14}, the first author showed that Ahlfors $Q$-regular Lipschitz differentiability spaces are strongly $Q$-unrectifiable under some natural assumptions. We now discuss a Carnot analog of this statement.

Namely, let us say that a metric space $X$ is \ti{strongly Carnot $Q$-unrectifiable} if $\mathcal{H}^Q(f(A)) = 0$ for any Lipschitz map $f \colon A \rightarrow \G$, defined on any subset $A \subset X$, mapping into any sub-Riemannian Carnot group $\G$. Since Euclidean spaces are sub-Riemannian Carnot groups, this is a stronger property than strong $Q$-unrectifiability.

The following is then immediate from Theorem \ref{blowup}.

\begin{corollary}\label{Carnotunrect}
Let $X$ be an Ahlfors $Q$-regular PI space such that, for almost every $x \in X$, no element of $\Tan(X,x)$ is bi-Lipschitz equivalent to a sub-Riemannian Carnot group. Then $X$ is strongly Carnot $Q$-unrectifiable.
\end{corollary}

Corollary \ref{Carnotunrect} applies, for example, to the topologically one-dimensional PI spaces constructed by Laakso \cite{La00} and Cheeger--Kleiner \cite{CK13_PI}. One can view Corollary \ref{Carnotunrect} as a strengthening, in the case of Ahlfors regular PI spaces, of Corollary 4.7 in \cite{DK16}, which says that a space satisfying the assumptions of Corollary \ref{Carnotunrect} admits no bi-Lipschitz embedding into any Carnot group.

\section{General criteria for bi-Lipschitz pieces of Lipschitz maps} \label{DCsec}

The previous section resolved Conjecture \ref{conjtangent} completely, but ultimately one would like to know whether Conjecture \ref{conjpiece} is true, i.e. whether such maps $f$ must be bi-Lipschitz on a set of positive measure. In this section, we give a sufficient condition for $f$ to decompose into countably many bi-Lipschitz pieces, expressed entirely in terms of its tangent maps (Proposition \ref{generalbilip}). In the following section, we will use this condition to prove Theorem \ref{piecethm}.

The following is the main result of this section.

\begin{prop} \label{generalbilip}
Let $M, N$ be complete, doubling metric spaces on which $\HH^Q$ is locally finite. Let $E \subset M$ be a measurable subset and $f \colon E \rightarrow N$ a Lipschitz map with image $Y = f(E) \subset N$. Suppose that
\begin{enumerate}[\normalfont (i)]
\item there are $C_0,R_0 > 0$ for which
$$C_0^{-1} r^Q \leq \mathcal{H}^Q(B(x,r)) \leq C_0 r^Q$$
and
$$C_0^{-1} r^Q \leq \mathcal{H}^Q(B(y,r)) \leq C_0 r^Q$$ 
for all $x \in E$, $y \in Y$ and $0<r<R_0$;
\item there is a measurable subset $E_0 \subset E$ such that $\mathcal{H}^Q(f(E \setminus E_0)) = 0$ and for each $x \in E_0$, every tangent $(\hat{E},\hat{x},\hat{f}) \in \Tan(E,x,f)$ surjects onto $\hat{Y}$.
\end{enumerate}
Then there are compact sets $E_i \subset E$, for $i \in \N$, such that $\mathcal{H}^Q(f(E \setminus \cup_i E_i)) = 0$ and $f|_{E_i}$ is bi-Lipschitz for each $i$.
\end{prop}

Let us emphasize here that $B(x,r)$ and $B(y,r)$ refer to balls in the ambient spaces $M$ and $N$, not their restrictions to the subsets $E$ and $Y$.

The proof of this proposition is based heavily on \cite{BL15}, in particular on the proof of Lemma 3.2 and especially on Theorem 5.3 therein. (The latter result itself builds on related work of David \cite{Da88}, Jones \cite{Jo88}, and Semmes \cite{Se00}.)
We restate it here, as we will use it directly, but first we need an important definition.

\begin{definition}[\cite{BL15}]\label{DCdef}
Let $(W,\mu)$ and $(Z,\HH^Q)$ be metric measure spaces (the latter equipped with $\HH^Q$ measure). Let $f \colon W \rightarrow Z$ be a Lipschitz mapping.

For constants $\beta,\epsilon,R>0$, we define the set
\begin{align*}
\text{DC}(\beta,\epsilon,R) = \{x\in W: \HH^Q &\left(B(f(x),\beta r) \cap f(B(x,r))\right) \geq\\
 &(1-\epsilon)\HH^Q(B(f(x),\beta r)), \forall r<R\}. 
\end{align*}

We say that $f$ \textit{satisfies David's condition} $\mu$-a.e. on $W$ if, for every $\epsilon>0$, 
\begin{equation}\label{DCaedef}
\mu\left(W \setminus \bigcup_{j=1}^\infty \bigcup_{k=1}^\infty \text{DC}\left(\frac{1}{j},\epsilon,\frac{1}{k}\right)\right) = 0.
\end{equation}
\end{definition}

The name ``David's condition'' was coined by Semmes \cite{Se00} for a more quantitative version of \eqref{DCaedef}, first discussed in \cite{Da88}. Roughly speaking, David's condition expresses that a map is ``almost locally surjective'' in a measure-theoretic sense.

David's condition is the key element in finding bi-Lipschitz pieces of Lipschitz mappings, as the following result of Bate and Li (building on \cites{Da88, Jo88, Se00}) shows. The following is Theorem 5.3 of \cite{BL15} (with the subsequent remarks incorporated).

\begin{theorem}\label{BLthm}
Let $(W,\mu)$ and $(Z,\HH^Q)$ be complete metric measure spaces, and let $f \colon W\rightarrow Z$ be Lipschitz. Assume that, for $\mu$-a.e. $x\in W$,
\begin{equation}\label{BLthm1}
0 < \liminf_{r \rightarrow 0} \frac{\mu(B(x,r))}{r^Q} \leq \limsup_{r \rightarrow 0} \frac{\mu(B(x,r))}{r^Q} < \infty
\end{equation}
and
\begin{equation}\label{BLthm2}
 \liminf_{r\rightarrow 0} \frac{\HH^Q(B(f(x),r) \cap f(W))}{r^Q} > 0.
\end{equation}
If $f$ satisfies David's condition $\mu$-a.e. on $W$, then there is a countable collection of compact sets $U_i \subseteq W$ such that $\mu(W \setminus \cup_i U_i) = 0$ and $f|_{U_i}$ is bi-Lipschitz for each $i$.
\end{theorem}

\begin{remark}
In proving Proposition \ref{generalbilip}, it will be convenient to apply Theorem \ref{BLthm} in the case where $W$ and $Z$ are not necessarily complete but only have compact completions.

In that case, we can still apply the theorem under one additional assumption. Suppose $W$, $Z$, and $f \colon W\rightarrow Z$ satisfy the assumptions of Theorem \ref{BLthm}, replacing the requirement that $W$ and $Z$ are complete with the assumption that their completions are compact. Consider the completions $(\overline{W}, \overline{\mu})$ and $(\overline{Z},\mathcal{H}^Q)$, where $\overline{\mu}$ is equal to $\mu$ on $W\subset \overline{W}$  and  $\overline{\mu}(\overline{W}\setminus W)=0$. Assume in addition that
$$ \mathcal{H}^Q(f(\overline{W})\setminus f(W))=0.$$

Then the assumption that $f\colon W\rightarrow Z$ satisfies David's condition $\mu$-a.e. immediately implies that $f \colon \overline{W}\rightarrow f(\overline{W})$ satisfies David's condition $\overline{\mu}$-a.e. Hence Theorem \ref{BLthm} applies to $f \colon \overline{W}\rightarrow f(\overline{W})$, and we obtain the desired bi-Lipschitz decomposition of $f$.
\end{remark}

The second result from \cite{BL15} that we will use concerns the construction of useful systems of dyadic ``cubes'' in general metric measure spaces. It is based on earlier results of David \cite{Da88} and Christ \cite{Ch90}.

\begin{prop}[\cite{BL15}, Proposition 2.1]\label{BLcubes}
Let $Z\subset N$ be a compact subset of a metric measure space $(N,\nu)$ for which there are $C,R >0$ with
$$ C^{-1} r^n \leq \nu(B(x,r)) \leq Cr^n$$
for all $x\in Z$ and $0<r<R$. Assume also that $Z \subset \overline{B}(y,r)$ for some $y\in Z$ and $0<r<R/16$.

Then there is a collection of subsets
$$ \Delta = \{Q^k_\omega \subset X :  k\in\mathbb{N}, \omega\in I_k\}$$
with the following properties:
\begin{enumerate}[\normalfont (i)]
\item $\nu(Z \setminus \cup_\omega Q^k_\omega)=0$;
\item if $\ell\geq k$ then either $Q^\ell_\alpha \subset Q^k_\omega$ or $Q^\ell_\alpha \cap Q^k_\omega = \emptyset$;
\item for each $Q^\ell_\alpha$ and $0\leq k \leq \ell$, there is a unique $\omega \in I_k$ with $Q^\ell_\alpha\subset Q^k_\omega$;
\item\label{cubeball} for each $Q^k_\omega$, there is $z^k_\omega\in Z$, called the ``center'' of $Q^k_\omega$, with
$$ B(z^k_\omega, 16^{-k-1}r) \subseteq Q^k_\omega \subseteq \overline{B}(z^k_\omega, 16^{-k+1}r); $$
\item each index set $I_k$ is finite.
\end{enumerate}
\end{prop}

Note that (\ref{cubeball}) implies, in particular, that $Q^k_{\omega} \cap Z \neq \emptyset$ for all $k \in \N$ and $\omega \in I_k$. We are now ready to prove Proposition \ref{generalbilip}.

\begin{proof}[Proof of Proposition \ref{generalbilip}]
For notational ease, let $\mu = \mathcal{H}^Q|_M$ and $\nu = \mathcal{H}^Q|_N$. By assumption, both measures are locally finite and, thus, inner regular. It is clear that we may assume $\nu(Y) >0$; otherwise there is nothing to prove. Finally, we may also assume that $f$ has Lipschitz constant equal to 1.

Let us first observe that for all $x \in E_0$, there are constants $L, R > 0$ such that $f(B(x,Lr) \cap E)$ is $r/10000$-dense in $B(f(x),r) \cap Y$ whenever $0< r< R$. Indeed, if this were to fail, then there would be a sequence $r_n \rightarrow 0$ and points $y_n \in B(f(x),r_n) \cap Y$ that are of distance at least $r_n/10000$ from $f(B(x,n r_n) \cap E)$. This implies that a tangent $(\hat{E},\hat{x},\hat{f}) \in \Tan(E,x,f)$ subordinate to some subsequence of $r_n$ would fail to surject onto $Y$, contrary to assumption.

Now, let $K \subset E_0$ be a compact set with $\nu(f(K)) > 0$ on which such constants $L,R >0$ exist and are uniform (we may assume also that $R < R_0$). Note that, by inner regularity, $E_0$ can be written, up to a set of measure zero, as a countable union of such compact sets $K$. As $\mathcal{H}^Q(f(E \setminus E_0)) = 0$, it suffices to obtain the desired conclusion for the restricted map $f|_K$. 

Let
\begin{align*}
 K' = \{x\in K: x &\text{ is a point of $\mu$-density of } K\subset M \text{ and }\\
 &f(x) \text{ is a point of $\nu$-density of } f(K)\subset N\},
\end{align*}
and note that $\nu(f(K) \setminus f(K'))=0$. 

Our goal is to verify David's condition (Definition \ref{DCdef}) for $f \colon K'\rightarrow f(K')$. More specifically, we will show that if $x \in K'$, then for each $\e >0$ there is $R_x \leq R$ such that
\begin{equation} \label{DCeq}
\nu( B(f(x),r) \cap f(K \cap B(x, 5Lr)) ) \geq (1-\e) \cdot \nu(f(K) \cap B(f(x),r))
\end{equation}
for all $0 < r< R_x$. Note that \eqref{DCeq} implies also that
$$\nu( B(f(x),r) \cap f(K' \cap B(x, 5Lr)) ) \geq (1-\e) \cdot \nu(f(K') \cap B(f(x),r))$$
for all $x \in K'$ and $0<r<R_x$. In particular, the set
$$ K' \setminus \bigcup_{k=1}^\infty \text{DC}\left(\frac{1}{5L},\epsilon,\frac{1}{k}\right) $$
has $\mu$-measure zero for each $\epsilon>0$, i.e., $f \colon K'\rightarrow f(K')$ satisfies David's condition.

Note, in addition, that $K'$ and $f(K')$ satisfy the requirements \eqref{BLthm1} and \eqref{BLthm2}, by our assumptions on $E$ and $Y$ and the definition of $K'$. Hence, by Theorem \ref{BLthm}, proving \eqref{DCeq} proves Proposition \ref{generalbilip}.

Now let us verify \eqref{DCeq}. Fix $x \in K$ a point of $\mu$-density with $f(x) \in f(K)$ a point of $\nu$-density. Let $\e >0$. Then there is $0< R_x < R /100L$ such that
\begin{equation}\label{muball}
\mu(B(x, 6Lr) \setminus K) \leq \e \cdot \mu(B(x, 6Lr))
\end{equation}
and
\begin{equation}\label{nuball}
\nu(B(f(x),r)) \leq 2 \cdot \nu(f(K) \cap B(f(x),r))
\end{equation}
for all $0 < r< R_x$. Finally fix $0<r<R_x$.

Fix a ``dyadic'' decomposition centered on the compact set 
$$Z = f(K) \cap \cl{B}(f(x), r) \subset N,$$
of the type in Proposition \ref{BLcubes}. Let $\Delta$ denote the resulting collection of subsets of $N$ (which we call ``cubes''). 

We now define a stopping time process based on this collection of cubes. Our stopping time process will begin with a ``top level'' cube $Q^1_\omega$ and descend scales one level at a time, identifying some sub-cubes $Q^k_\alpha$ at which the process terminates and continuing to descend to sub-cubes when the process does not terminate. We will then run the same process on each of the finitely many top level cubes $Q^1_\omega$.

Fix one of the top level cubes and call it $Q^1$. Let $x^1=x$. Begin with the triple $(x^1, Q^1, 2r)$. As $f(B(x^1,2Lr) \cap E)$ is $r/5000$-dense in $B(f(x^1),2r) \cap Y$, which contains $Q^1\cap Z$,  at least one of the following is true:
\begin{enumerate}
\item for each center $z_{\omega}^2 \in Q^1$, there is $x_{\omega}^2 \in K \cap B(x^1,3Lr)$ with $d_N(f(x_{\omega}^2), z_{\omega}^2) < r/1000$;
\item there is a point $u \in B(x^1,2Lr) \cap E$ and a center $z_{\omega}^2 \in Q^1$ for which $d_N(f(u),z_{\omega}^2) < r/5000$ and $B(u, r/5000) \cap K = \emptyset$.
\end{enumerate}
We stop the process at $Q^1$ if (1) fails. Otherwise, we repeat the process with each of the triples $(x_{\omega}^2, Q_{\omega}^2, 2r/16)$. 

For the general step, suppose that the process has discovered a point $x_{\omega}^k \in K$ for which $d_N(f(x_{\omega}^k),z_{\omega}^k) < r/(1000 \cdot 16^{k-2})$. We consider the triple $(x_{\omega}^k, Q_{\omega}^k, 2r/16^{k-1})$. As $f(B(x_{\omega}^k, 2Lr/16^{k-1}) \cap E)$ is $r/(5000 \cdot 16^{k-1})$-dense in $B(f(x_{\omega}^k), 2r/16^{k-1}) \cap Y$, and since
$$Q_{\omega}^k \cap Z \subset \cl{B}(z_{\omega}^k, r/16^{k-1}) \cap Z \subset B(f(x_{\omega}^k), 2r/16^{k-1}) \cap Y,$$
at least one of the following is true:
\begin{enumerate}
\item for each center $z_{\alpha}^{k+1} \in Q_{\omega}^k$, there is $x_{\alpha}^{k+1} \in K \cap B(x_{\omega}^k, 3Lr/16^{k-1})$ with 
$$d_N(f(x_{\alpha}^{k+1}), z_{\alpha}^{k+1}) < r/(1000 \cdot 16^{k-1});$$
\item there is a point $u \in E \cap B(x_{\omega}^k, 2Lr / 16^{k-1})$ and a center $z_{\alpha}^{k+1} \in Q_{\omega}^k$ for which 
$$d_N(f(u), z_{\alpha}^{k+1}) < r/(5000 \cdot 16^{k-1})$$ 
and $B(u, r/5000 \cdot 16^{k-1}) \cap K = \emptyset$.
\end{enumerate}
We stop the process at $Q_{\omega}^k$ if (1) fails. Otherwise, we repeat the process with each of the triples $(x_{\alpha}^{k+1}, Q_{\alpha}^{k+1}, 2r/16^k)$. Observe that if the process terminates at $Q_{\omega}^k$, then for $u$ the point in alternative (2), we have
\begin{equation} \label{imageinclusionQ}
f( B(u, r/5000 \cdot 16^{k-1}) \cap E) \subset B( z_{\alpha}^{k+1}, r/ 1000 \cdot 16^{k-1} ) \cap Y \subset Q_{\alpha}^{k+1} \subset Q_{\omega}^k.
\end{equation}
Finally, let us note that all of the points $x_{\omega}^k \in K$ that this process discovers have
$$d_M(x,x_{\omega}^k) = d_M(x^1,x_{\omega}^k) \leq 3Lr + 3Lr/16 + \ldots + 3Lr/16^{k-1} < 4Lr,$$
and so lie in the ball $B(x, 4Lr)$. 

This completes the description of the stopping time process that begins with a single top cube $Q^1$. We run this same stopping time process on each of the finitely many top cubes $Q^1_\omega$.

Let $\{Q_i\}$ be the collection of all the cubes at which the process (beginning at any top cube) terminates, and note that they are pairwise disjoint. Let us first verify that the total volume of these cubes is small. If the process terminates at $Q_i = Q_{\omega}^k$, then there is a point $u_i \in B(x,5Lr) \cap E$ for which 
$$B_i = B(u_i, r/10000 \cdot 16^{k-1})$$ 
is disjoint from $K$, and for which $f(2B_i \cap E) \subset Q_i$ by \eqref{imageinclusionQ}. As the collection $\{Q_i\}$ is pairwise disjoint, it follows that $\{B_i\}$ is also pairwise disjoint. Indeed, if $B_i \cap B_j \neq \emptyset$ for some $i \neq j$ with radii $r_i \geq r_j$, then we would have $u_j \in 2B_i \cap E$, implying that $f(u_j) \in Q_i \cap Q_j$. 

Using our measure theoretic assumptions on $E$ and $Y$, we know that
$$\mu(B_i) \gtrsim (r/16^k)^Q \gtrsim \nu(Q_i)$$
for each $i$ (with uniform constants), and by \eqref{muball} we can estimate
$$\nu(\cup_i Q_i) = \sum_i \nu(Q_i) \lesssim \sum_i \mu(B_i) \lesssim \mu( B(x,6Lr) \setminus K) \lesssim \e \cdot \mu(B(x, 6Lr)).$$
As 
$$\mu(B(x,6Lr)) \lesssim r^Q \lesssim \nu(B(f(x),r)) \lesssim \nu(f(K) \cap B(f(x),r)),$$
by \eqref{nuball} and our assumptions, we see that
$$\nu(\cup_i Q_i) \leq C' \e \cdot \nu(f(K) \cap B(f(x),r))$$
where $C'$ is a uniform constant.

To finish the verification of \eqref{DCeq}, consider a point $y \in Z \setminus \cup_i Q_i$ that lies in $\cup_{\omega} Q_{\omega}^k$ for each $k$. Almost every $y \in Z \setminus \cup_i Q_i$ satisfies this. As the stopping-time process did not terminate on any cube that contains $y$, we know that for each $k \in \N$, there are $x_{\omega_k}^k \in K \cap B(x, 4Lr)$ for which
$$d_N(f(x_{\omega_k}^k), y) \leq d_N(f(x_{\omega_k}^k), z_{\omega_k}^k) + d_N(z_{\omega_k}^k, y) \leq 2r/16^{k-1}.$$
Using that $K$ is compact, we see that $y \in f(K \cap B(x, 5Lr))$. Hence,
$$\begin{aligned}
\nu ( f(K) \cap B(f(x),r) \setminus f(K \cap B(x, 5Lr)) ) &\leq \nu(\cup_i Q_i)  \\
&\leq   C' \e \cdot \nu(f(K) \cap B(f(x),r)),
\end{aligned}$$
which suffices to show \eqref{DCeq}.
\end{proof}

\section{Bi-Lipschitz pieces for maps into Carnot groups} \label{piecesec}

We now use the criteria for bi-Lipschitz pieces established in the previous section, Proposition \ref{generalbilip}, to prove Theorem \ref{piecethm}. Let us emphasize that the conditions under which a bi-Lipschitz decomposition exists are expressed entirely in terms of the image set $f(A)$, not on properties of the mapping. We re-state the theorem now for the reader's convenience.

\begin{piecetheorem}
Let $X$ be an Ahlfors $Q$-regular PI space, $A \subset X$ a compact subset, and $\G$ be a sub-Riemannian Carnot group. Suppose $f \colon A \rightarrow \G$ is a Lipschitz mapping such that $\mathcal{H}^Q(f(A)) > 0$, and let $\nu = \mathcal{H}^Q|_{f(A)}$. Then the following are equivalent.

\begin{enumerate}[\normalfont (i)]
\item There are countably many compact subsets $A_i \subset A$ such that $\nu(f(A \setminus \cup_i A_i)) = 0$
and $f$ is bi-Lipschitz on each $A_i$.
\item At $\nu$-a.e. point $y \in f(A)$, each tangent measure $\hat{\nu} \in \Tan(\nu,y)$ is comparable to the restriction of $\mathcal{H}^Q$ to a Carnot subgroup of $\G$.
\item At $\nu$-a.e. point $y \in f(A)$, the support of each tangent measure $\hat{\nu} \in \Tan(\nu,y)$ is a connected subset of $\G$.
\end{enumerate}

\end{piecetheorem}

For the remainder of this section, we fix $f \colon A \rightarrow \G$ as in the statement of Theorem \ref{piecethm}, and let $\nu = \mathcal{H}^Q|_{f(A)}$. Let us begin by establishing the following preliminary lemma.

\begin{lemma} \label{density}
For $\nu$-a.e. point $y \in f(A)$, we have
$$0 < \liminf_{r \rightarrow 0} \frac{\nu(B(y,r))}{r^Q} \leq \limsup_{r \rightarrow 0} \frac{\nu(B(y,r))}{r^Q} \leq 1.$$
In particular, $\nu$ is pointwise doubling.
\end{lemma}

\begin{proof}
We treat the lower bound first, essentially following the argument in Theorem 7.9 of \cite{Mat95}. Given $\e >0$, let 
$$E_{\e} = \{ y \in f(A) : \liminf_{r \rightarrow 0} \frac{\nu(B(y,r))}{r^Q} < \e \},$$
and let $U \subset X$ be the open 1-neighborhood of $A$. We wish to show that
$$\nu(E_{\e}) \lesssim \e \mathcal{H}^Q(U),$$
from which the desired result follows immediately. To do this, note that for each $y \in E_{\e}$ we can find a radius $0< r_y <1$ small enough that $\nu(B(y,r_y)) < \e r_y^Q$. Applying the basic covering lemma \cite[Theorem 1.2]{He01} to the collection of balls 
$$\{B(y,r_y/5): y \in E_{\e} \}$$ 
in $\G$, we can find a countable disjoint sub-collection $\{B(y_i,r_i/5)\}_i$ for which the collection $\{B(y_i,r_i)\}_i$ covers $E_{\e}$. For each $i$, choose $x_i \in f^{-1}(y_i)$ and note that the balls $\{B(x_i,r_i/10L)\}_i$ in $U$ are pairwise disjoint, where $L$ is the Lipschitz constant of $f$. We then have
$$\nu(E_{\e}) \leq \sum_i \nu(B(y_i,r_i)) \leq \e \sum_i r_i^Q \lesssim \e \sum_i \mathcal{H}^Q(B(x_i,r_i/10L)) \lesssim \e \mathcal{H}^Q(U),$$
where we have used the fact that $X$ is Ahlfors $Q$-regular.

Now let us verify the upper bound, following the argument in Theorem 1.3.9 of \cite{LY02}. For $t >1$, define
$$E_t = \{ y \in f(A) : \limsup_{r \rightarrow 0} \frac{\nu(B(y,r))}{r^Q} > t\}.$$
For $\e >0$, let $U \subset \G$ be open with $E_t \subset U$ and $\nu(U) \leq \nu(E_t) + \e$. Finally, let $\delta>0$ be small and consider the collection of balls
$$\mathcal{B} = \{B(y,r) : y \in E_t, 0<r<\delta/5, B(y,r) \subset U, \text{ and } \nu(B(y,r)) > t r^Q \}.$$
This is a fine cover of $E_t$, in the sense that each $y \in E_t$ is contained in balls of $\mathcal{B}$ with arbitrarily small radius. A slight variation of the standard covering lemma (see \cite[Corollary 1.3.3]{LY02}) gives a countable disjoint sub-collection $\{B_i = B(y_i,r_i)\}_i$ such that
$$E_t \subset \bigcup_{i=1}^m B_i \cup \bigcup_{i=m+1}^{\infty} 5B_i$$
for each $m \in \N$. Recalling the definition of $\HH^Q$ from Section \ref{basicdefs}, we bound
$$\mathcal{H}_{\delta}^Q(E_t) \leq \sum_{i=1}^m r_i^Q + \sum_{i=m+1}^{\infty} (5r_i)^Q 
\leq \frac{1}{t} \sum_{i=1}^m \nu(B_i) + \frac{5^Q}{t} \sum_{i=m+1}^{\infty} \nu(B_i).$$
As the $B_i$'s are disjoint and contained in the finite-measure set $U$, we know that $\sum_{i=m+1}^{\infty} \nu(B_i) \rightarrow 0$ as $m \rightarrow \infty$. This gives
$$\mathcal{H}_{\delta}^Q(E_t) \leq \frac{1}{t} \sum_{i=1}^{\infty} \nu(B_i) \leq \frac{1}{t} \nu(U) \leq \frac{1}{t}( \nu(E_t) + \e).$$
Sending $\delta \rightarrow 0$ shows that $\nu(E_t) \leq (\nu(E_t) + \e)/t$, and then sending $\e \rightarrow 0$ gives $\nu(E_t) \leq\nu(E_t) / t$. As $t >1$, we must have $\nu(E_t) = 0$. The lemma follows.
\end{proof}

We now decompose the domain and range of our mapping $f$ into sets on which the measure-theoretic behavior is quantitatively controlled.

By Theorem \ref{blowup}, there is a set $E_0 \subset A$ with $\nu(f(A \setminus E_0)) = 0$ such that, for all $x \in E_0$, there is a Carnot subgroup $\H_x\subset \G$ with $\Hdim(\H_x) = Q$ for which every $(\hat{X}, \hat{x}, \hat{f}) \in \Tan(A,x,f)$ is a Lipschitz quotient map onto $\H_x$.

By Lemma \ref{density}, along with inner regularity of $\nu$, we can write $f(A)$, up to a set of measure zero, as a countable union of compact subsets $Y_k \subset f(A)$ with the following property. For each $k \in \N$, there is $\alpha_k >0$ such that
\begin{equation}\label{Yk}
\alpha_k r^Q \leq \nu(B(y,r)) \leq 2r^Q
\end{equation}
for all $0< r < \alpha_k$ and all $y \in Y_k$.

Now, for $k \in \N$, let $E_k \subset f^{-1}(Y_k) \cap E_0$ be the set of $\mathcal{H}^Q$-density points for $f^{-1}(Y_k) \cap E_0$ in $X$. These choices give the following.
\begin{itemize}
\item $\nu(f(A \setminus \cup_k E_k)) = 0$;
\item $f(E_k) \subset Y_k$ for each $k \in \N$;
\item for every $x \in E_k$ and every $(\hat{X},\hat{x},\hat{f})\in\Tan(A,x,f)$, the map $\hat{f}$ is a Lipschitz quotient onto $\H_x$;
\item the uniform upper and lower estimates \eqref{Yk} hold all points $y\in Y_k$.
\end{itemize}
In what follows, we will repeatedly use Lemmas \ref{subsettangentmeasure}, \ref{subsettangent}, and \ref{densitytangent}.

\begin{lemma} \label{tangent1}
Let $y \in Y_k$ be a point of $\nu$-density for $Y_k$. Suppose that $\hat{\nu} \in \Tan(\nu,y)$ and $\hat{Y}_k \in \Tan_{\G}(Y_k,y)$ are subordinate to a common sequence $\lambda_i \rightarrow 0$. Then $\supp(\hat{\nu}) = \hat{Y}_k$. 

In particular, if $x \in E_k \cap f^{-1}(y)$, then each $\hat{\nu} \in \Tan(\nu,y)$ has $\H_x \subset \supp(\hat{\nu})$.
\end{lemma}

\begin{proof}
After passing to a subsequence, by Lemma \ref{tanmeasures}, we may assume that
$$\hat{\nu} = c_1 \cdot \lim_{i \rightarrow \infty} \nu(B(y,\lambda_i))^{-1} (T_{y,\lambda_i})_{\#} \nu$$
for some $c_1>0$. Then, using Lemma \ref{subsettangentmeasure} and Lemma \ref{tanmeasures} again, we may pass to a further subsequence so that
$$\hat{\nu} = c_2 \cdot \lim_{i \rightarrow \infty} \nu(B(y,\lambda_i) \cap Y_k)^{-1} (T_{y,\lambda_i})_{\#} (\nu|_{Y_k})$$
for some $c_2 >0$. By assumption, we still have
$$\hat{Y}_k = \lim_{i \rightarrow \infty} T_{y,\lambda_i}(Y_k) \subset \G.$$
Our goal is to show that $\hat{Y}_k = \supp(\hat{\nu})$.

First, suppose that $z \notin \hat{Y}_k$. Then there is $r > 0$ such that 
$$T_{y,\lambda_i}(Y_k) \cap B(z,r) = \emptyset$$ 
for all $i$. This means that $(T_{y,\lambda_i})_{\#} (\nu|_{Y_k}) (B(z,r)) = 0$. Consequently,
$$\hat{\nu}(B(z,r)) \leq \liminf_{i \rightarrow \infty} \nu(B(y,\lambda_i)\cap Y_k)^{-1} (T_{y,\lambda_i})_{\#} (\nu|_{Y_k}) (B(z,r)) = 0,$$
using \cite[Theorem 1.24]{Mat95}, and we find that $z \notin \supp(\hat{\nu})$.

Now we prove the reverse containment, i.e., that $\hat{Y}_k \subset \supp(\hat{\nu})$. To this end, fix $z \in \hat{Y}_k$. Then there is a sequence of points $y_i \in Y_k$ such that $T_{y,\lambda_i}(y_i) \rightarrow z$. For $r > 0$, we have, again using \cite[Theorem 1.24]{Mat95},
$$\hat{\nu}(\cl{B}(z,r)) \geq \limsup_{i \rightarrow \infty} \nu(B(y,\lambda_i))^{-1} \nu(\cl{B}(y \cdot \delta_{\lambda_i}(z), \lambda_i r)).$$
Using that $T_{y,\lambda_i}(y_i) \rightarrow z$, we have $d_{cc}(\delta_{\lambda_i^{-1}}(y^{-1} \cdot y_i), z) \rightarrow 0$, and by homogeneity,
$$d_{cc}(y_i, y \cdot \delta_{\lambda_i}(z)) = \lambda_i d_{cc}(\delta_{\lambda_i^{-1}}(y^{-1}\cdot y_i), z) = o(\lambda_i).$$
In particular, for large enough $i$, we have
$$B(y_i, \lambda_i r/2) \subset B(y\cdot \delta_{\lambda_i}(z), \lambda_i r),$$
so that,
$$\hat{\nu}(\cl{B}(z,r)) \geq \limsup_{i \rightarrow \infty}  \nu(B(y,\lambda_i))^{-1}  \nu(B(y_i,\lambda_i r/2)) \gtrsim_k r^Q.$$
The last inequality here comes from the fact that $y,y_i \in Y_k$. This shows that $z \in \supp(\hat{\nu})$, as desired.

To prove the second statement of the lemma, fix $\hat{\nu} \in \Tan(\nu,y)$ subordinate to the sequence $\lambda_i \rightarrow 0$. After passing to a subsequence, there is 
$$(\hat{X}, \hat{x}, \hat{f}) \in \Tan(E_0 \cap f^{-1}(Y_k),x,f)=\Tan(A,x,f),$$
defined along the same sequence of scales, with $\hat{f} \colon \hat{X} \rightarrow \hat{Y}_k \subset \G$. As $x \in E_0$, we know that $\hat{f}(\hat{X}) = \H_x$. Using the first part of the lemma, we conclude that $\H_x \subset \hat{Y}_k = \supp(\hat{\nu})$.
\end{proof}

\begin{lemma} \label{tangent2}
For $\nu$-a.e. point $y \in Y_k$, the following holds for each $x \in E_k \cap f^{-1}(y)$. For each tangent measure $\hat{\nu} \in \Tan(\nu,y)$ and each point $z \in \supp(\hat{\nu})$, we have $z \cdot \H_x \subset \supp(\hat{\nu})$.
\end{lemma}

\begin{proof}
By Lemma \ref{density}, we know that $\nu$ is pointwise doubling. Thus, Proposition \ref{mattilamove} guarantees that for $\nu$-a.e. $y \in Y_k$, for each $\hat{\nu} \in \Tan(\nu,y)$, it holds that $(T_{z,1})_\# \hat{\nu} \in \Tan(\nu, y)$ for all $z \in \supp(\hat{\nu})$. Fix such a point $y$ that is also a $\nu$-density point for $Y_k$, and let $x \in E_k \cap f^{-1}(y)$.

Given $\hat{\nu} \in \Tan(\nu,y)$, for each $z \in \supp(\hat{\nu})$ we have $(T_{z,1})_\# \hat{\nu} \in  \Tan(\nu, y)$. Lemma \ref{tangent1} then ensures that
$$\H_x \subset \supp((T_{z,1})_\# \hat{\nu}) = z^{-1} \cdot \supp(\hat{\nu}).$$
We therefore conclude that $z \cdot \H_x \subset \supp(\hat{\nu})$.
\end{proof}

\begin{lemma} \label{finite}
For $\nu$-a.e. point $y \in Y_k$, the following holds for each $x \in E_k \cap f^{-1}(y)$. For each tangent measure $\hat{\nu} \in \Tan(\nu,y)$, there are $z_1, \ldots ,z_m \in \G$ such that
$$\supp(\hat{\nu}) = z_1 \cdot \H_x \cup \cdots \cup z_m \cdot \H_x$$
where the $z_i \cdot \H_x$ are disjoint translates of the Carnot subgroup $\H_x$.
\end{lemma}

\begin{proof}
Let $y \in Y_k$ be a point of $\nu$-density of $Y_k$ at which the statements in Lemma \ref{tangent2} and Proposition \ref{mattilamove} hold. This constitutes a set of full measure in $Y_k$. Fix $x \in E_0 \cap f^{-1}(y)$ and $\hat{\nu} \in \Tan(\nu,y)=\Tan(\nu|_{Y_k},y)$. Then, for each $z \in \supp(\hat{\nu})$, we have $z \cdot \H_x \subset \supp(\hat{\nu})$. Suppose that $z_1,\ldots,z_m \in \supp(\hat{\nu})$ are distinct points for which the translates $z_i \cdot \H_x$ are disjoint. We wish to show that there is an upper bound on the number $m$ of such points.

To this end, let 
$$t = \left(\liminf_{r \rightarrow 0} \nu(B(y,r))r^{-Q}\right) / \left(\limsup_{r \rightarrow 0} \nu(B(y,r))r^{-Q}\right) \geq \alpha_k/2 > 0.$$

 Fix a large radius $R > \max\{d_{cc}(0,z_i) : 1\leq i \leq m\}$, and let $\mu = (\delta_{R^{-1}})_{\#} \hat{\nu}$. By Proposition \ref{mattilamove}, we know that $\mu \in \Tan(\nu,y)$ as well. Consequently, by Lemma \ref{tanmeasureAR}, there is a constant $c > 0$ for which
\begin{equation}\label{muAR}
t r^Q \leq c^{-1} \mu(B(z,r)) \leq r^Q
\end{equation}
for all $z \in \supp(\mu) = \delta_{R^{-1}}(\supp(\hat{\nu}))$ and all $0<r<\infty$. By a standard fact about Hausdorff measure, this implies that there are constants $c_1, c_2 >0$ depending only on $t, Q$ such that
$$c_1 \mathcal{H}^Q(F) \leq c^{-1} \mu(F) \leq c_2 \mathcal{H}^Q(F)$$
for all Borel subsets $F \subset \supp(\mu)$.

Now, we observe that on the one hand, $\mu(B(0,2)) \leq 2^Q c$, by \eqref{muAR}. On the other hand, for each $1 \leq i \leq m$, we have 
$$\delta_{R^{-1}}(z_i) \cdot \H_x \subset \supp(\mu)$$ 
and $d_{cc}(\delta_{R^{-1}}(z_i),0) < 1$, so $B(0,2)$ contains $m$ disjoint translates of the ball $B(0,1)\cap \H_x$. Thus, the comparability of $\mu$ with $\mathcal{H}^Q|_{\supp(\mu)}$ gives the lower bound
$$\mu(B(0,2)) \geq c_1 c \cdot \mathcal{H}^Q(B(0,2) \cap \supp(\mu)) \geq m c_1 c \cdot \mathcal{H}^Q(B(0,1)\cap \H_x).$$
Thus, from this and \eqref{muAR} we obtain 
$$m \leq \frac{2^Q}{c_1\mathcal{H}^Q(B(0,1)\cap \H_x)}<\infty,$$
as desired.
\end{proof}

We are now ready to begin the proof of the equivalences in Theorem \ref{piecethm}

\begin{proof}[\textbf{Proof of Theorem \ref{piecethm}}]
Let us continue using the notation from above, in particular the sets $E_0\subset A$, $E_k \subset A$, and $Y_k \subset f(A)$.

(i) $\Rightarrow$ (ii). For $\nu$-a.e. point $y \in f(A)$, there are $k, i \in \N$ for which the following are true:
\begin{itemize}
\item $y \in Y_k$ and the conclusion of Lemma \ref{finite} holds;
\item $y = f(x)$ for some $x \in E_k \cap A_i$;
\item $x$ is a $\mu$-density point for $E_k \cap A_i$, and $y$ is a $\nu$-density point for $f(E_k \cap A_i)$.
\end{itemize}
Fix such $x\in A$, $y\in f(x)$, and indices $k$ and $i$. Let $G_{k,i} = f(E_k\cap A_i)\subset Y_k\subset f(A)$.

Take any tangent measure $\hat{\nu} \in \Tan(\nu,y)$. By passing to subsequences, we can also consider tangent objects
$$(\hat{X},\hat{x},\hat{f}) \in \Tan(E_k \cap A_i,x,f),$$
$$\hat{G}_{k,i} \in \Tan_{\G}(G_{k,i},y), \text{ and }$$
$$ \hat{Y}_{k} \in \Tan_{\G}(Y_k,y),$$
all subordinate to the same sequence of scales as $\hat{\nu}$.

Then $\hat{f} \colon \hat{X} \rightarrow \hat{G}_{k,i}$ is a bi-Lipschitz map onto $\hat{G}_{k,i}$, which coincides with $\hat{Y}_k$ by Lemma \ref{subsettangent}. At the same time, we know that $\hat{f}(\hat{X}) = \H_x$ because $x \in E_0$. Thus, appealing to Lemma \ref{tangent1} gives
$$\supp(\hat{\nu}) = \hat{Y}_k = \hat{f}(\hat{X}) = \H_x.$$
To conclude, we note that by Lemma \ref{tanmeasureAR}, there are constants $c_1,c_2 > 0$ for which
$$c_1 r^Q \leq \hat{\nu}(B(z,r)) \leq c_2 r^Q$$
for all $z \in \supp(\hat{\nu}) = \H_x$ and all $0<r<\infty$. This implies that $\hat{\nu}$ is comparable to $\mathcal{H}^Q|_{\H_x}$, as desired.

(ii) $\Rightarrow$ (iii). This is immediate because Carnot subgroups are connected.

(iii) $\Rightarrow$ (i). It suffices to verify that the hypotheses of Proposition \ref{generalbilip} are satisfied for $M = X$, $N = f(A)$, $E=E_k$, and the map $f|_{E_k}$, where $k \in \N$ is fixed. Indeed, that proposition then gives bi-Lipschitz decomposition for each $f|_{E_k}$, and we already know that $\nu(f(A\setminus \cup_k E_k)) = 0$.

From the Ahlfors $Q$-regularity of $X$ and the fact that $f(E_k)$ is in $Y_k$ which satisfies \eqref{Yk}, the first hypothesis of Proposition \ref{generalbilip} is satisfied. To verify the second hypothesis, let $G_k=f(E_k)$ and let $E'_k \subset E_k$ be the set of points $x$ for which the following are true:
\begin{itemize}
\item $x$ is a point of $\mathcal{H}^Q$-density for $E_k$ in $X$, and $f(x)$ is a point of $\nu$-density for $G_k$ in $f(A)$;
\item Lemma \ref{finite} and (iii) hold at $y=f(x) \in Y_k$.
\end{itemize}
For $x \in E'_k$ fixed, let $y=f(x)$. It follows from Lemma \ref{finite} and assumption (iii) at $y$ that each $\hat{\nu} \in \Tan(\nu,y)$ has $\supp(\hat{\nu}) = \H_x$. Indeed, the only way a non-trivial union of $m$ translates of $\H_x$ can be connected is if $m=1$.

Now consider any tangent $(\hat{X}, \hat{x}, \hat{f}) \in \Tan(E_k,x,f)=\Tan(A,x,f)$ of the restricted map $f|_{E_k}$. By passing to subsequences, we may also obtain (subordinate to the same sequence of scales), tangent objects
$$ \hat{\nu} \in \Tan(\nu,y),$$
$$ \hat{Y}_k \in \Tan_{\G}(Y_k, y),\text{ and }$$
$$ \hat{G}_k \in \Tan_{\G}(G_k,y).$$

We then have
$$\hat{f}(\hat{X}) = \H_x = \supp(\hat{\nu}) = \hat{Y}_k = \hat{G}_k,$$
where the first equality holds because $x\in E_0$, the second as remarked above, the third by Lemma \ref{tangent1}, and the fourth by Lemma \ref{subsettangent}. Thus, we find that the tangent map surjects onto $\hat{G}_k$, as desired.

This verifies the hypotheses of Proposition \ref{generalbilip} for each $f|_{E_k}$, and hence proves that (i) holds for $f$.
\end{proof}

The following is an immediate corollary of Theorem \ref{piecethm}.

\begin{corollary}
Let $X$ be a $Q$-regular PI space, and let $A \subset X$ be a compact subset of positive measure. Suppose $f \colon A \rightarrow \G$ is a Lipschitz map into a sub-Riemannian Carnot group $\G$ with $\Hdim(\G) = Q$ such that $\mathcal{H}^Q(f(A)) >0$. Then there are countably many subsets $A_i \subset A$ such that $\mathcal{H}^Q(f(A \setminus \cup_i A_i)) = 0$
and $f|_{A_i}$ is bi-Lipschitz for each $i$.
\end{corollary}

\begin{proof}
If $\nu = \mathcal{H}^Q|_{f(A)}$, then for every $\nu$-density point $y$ of $f(A)$, each tangent measure $\hat{\nu} \in \Tan(\nu, y)$ has $\supp(\hat{\nu}) = \G$.
\end{proof}

One might hope to obtain bi-Lipschitz pieces in general (i.e. without the restriction $\Hdim(\G) = Q$) by post-composing the map $f \colon A \rightarrow \G$ with an appropriately-chosen ``projection" onto a Carnot subgroup $\H_x$ of dimension $Q$, and then apply the above corollary. This is problematic, though, when $\G$ is not Euclidean: the natural linear projections may not be Lipschitz in the Carnot--Carath\'eodory metrics, and it may be that no good analogs of Lipschitz projections exist. For example, there is no Lipschitz map from the second Heisenberg group $\H^2$ to the first Heisenberg group $\H^1$ that has full rank on a set of positive measure. Indeed, if such a map were to exist, Pansu's differentiability theorem would give a surjective group homomorphism from $\H^2$ to $\H^1$ that sends horizontal vectors to horizontal vectors. It can be easily checked, though, that horizontal homomorphisms from $\H^2$ to $\H^1$ have rank at most 1.

However, it is plausible that this approach could work for mappings from PI spaces into Euclidean space, i.e., in the setting of Theorem \ref{Rnpiecethm}. In the next section, we prove this result, though we rely on a quite different (and shorter) argument.

\section{Bi-Lipschitz pieces for maps into Euclidean spaces}\label{Rnpiecesection}

In this section, we present the brief proof of Theorem \ref{Rnpiecethm}. The idea is to first show that, under the assumptions of the theorem, the set $A$ must in fact be $n$-rectifiable; this uses two results from \cite{GCD14} and \cite{BL15}. This step reduces Theorem \ref{Rnpiecethm} to the classical bi-Lipschitz pieces statement for mappings between subsets of Euclidean space (\cite[Lemma 3.2.2]{Fe69}). The arguments in this section are essentially independent of the rest of the paper.

In fact, as the proof shows, Theorem \ref{Rnpiecethm} works in the broader setting of Lipschitz differentiability spaces, and not just RNP-differentiability spaces or subsets of PI spaces. (Recall the discussion in Remark \ref{assumptions} about the various implications among these properties.)

We need to recall the notion of strong unrectifiability, introduced in subsection \ref{carnotunrectsection}: A metric space $X$ is said to be \textit{strongly $Q$-unrectifiable} if $\HH^Q(f(X))=0$ for all $N\in\mathbb{N}$ and all Lipschitz mappings $f \colon X \rightarrow \R^N$.

The first ingredient in the proof of Theorem \ref{Rnpiecethm} is a result from \cite{GCD14} which shows strong unrectifiability of Lipschitz differentiability spaces under a dimension condition.
\begin{theorem}[\cite{GCD14}, Theorem 1.4]\label{strongunrect}
Let $X$ be an Ahlfors $Q$-regular Lipschitz differentiability space containing a chart $U$ of dimension $k$, with $k<Q$. Then $U$ is strongly $Q$-unrectifiable.
\end{theorem}

The other ingredient in this proof of Theorem \ref{Rnpiecethm} is a characterization of rectifiable metric measure spaces due to Bate and Li \cite{BL15}. We quote only a piece of Theorem 1.2 from \cite{BL15}. In this context, it can be viewed as a sort of converse to Theorem \ref{strongunrect}.

\begin{theorem}[\cite{BL15}, Theorem 1.2]\label{bateli}
Let $(X,d,\mu)$ be a metric measure space such that
\begin{equation}\label{densities}
0 < \liminf_{r\rightarrow 0} \frac{\mu(B(x,r))}{r^n} \leq \limsup_{r\rightarrow 0} \frac{\mu(B(x,r))}{r^n} < \infty
\end{equation}
for $\mu$-a.e. $x\in X$ and some $n\in\mathbb{N}$. Suppose further that $(X,d,\mu)$ is a Lipschitz differentiability space with a chart $U$ of dimension $n$. Then $\mu|_U$ is $n$-rectifiable.
\end{theorem}

Recall that a measure $\mu$ on a metric space $Z$ is called \textit{$n$-rectifiable} if there are countably many compact sets $F_i\subseteq \RR^n$ and Lipschitz mappings $f_i \colon F_i\rightarrow X$ such that 
$$ \mu(X\setminus \cup_i f_i(F_i)) = 0.$$
By a result of Kirchheim \cite[Lemma 4]{Ki94}, if $\mathcal{H}^n$ is $n$-rectifiable on a metric space $Z$, then in fact there are countably many compact sets $E_i\subseteq \RR^n$ and bi-Lipschitz mappings $g_i \colon E_i\rightarrow X$ such that 
$$ \HH^n(X\setminus \cup_i g_i(E_i)) = 0.$$

\begin{proof}[\textbf{Proof of Theorem \ref{Rnpiecethm}}]
Let $X$ be an Ahlfors $Q$-regular PI space, $A$ a compact subset of $X$, and $N\in\mathbb{N}$. Suppose $f \colon A\rightarrow \RR^N$ is a Lipschitz mapping such that
$$\mathcal{H}^Q(f(A)) > 0.$$

Since we can write $A = \bigcup_i (A\cap U_i)$, where $U_i$ are the differentiability charts of $X$, it suffices to prove the theorem in the case that $A$ is contained in a single differentiability chart $U$.

In that case, we may extend $f$ to all of $U$ by McShane's extension theorem \cite[Theorem 6.2]{He01}, in which case
$$ \HH^Q(f(U)) \geq \HH^Q(f(A)) > 0.$$
It follows that $U$ is not strongly $Q$-unrectifiable. Theorem \ref{strongunrect} then implies that $Q$ is an integer and $U$ is a $Q$-dimensional differentiability chart.

Since $X$ is Ahlfors $Q$-regular, the metric measure space $(U, d, \HH^Q|_U)$ satisfies the density assumptions \eqref{densities}. Therefore $\HH^Q|_U$ is $Q$-rectifiable by Theorem \ref{bateli}.

It follows that we can write $A = Z \cup \bigcup_i g_i(E_i)$, where $\HH^Q(Z)=0$, the countably-many $E_i$ are compact subsets of $\RR^Q$, and the maps $g_i \colon E_i\rightarrow X$ are bi-Lipschitz.

We now use the classical Euclidean results to tell us that, for each $i$, the set $E_i$ can be written as a countable union $Z_i \cup \bigcup_{j} E_{i,j}$, where $\HH^Q(f(g_i(Z_i)))=0$ and $f\circ g_i$ is bi-Lipschitz on $E_{i,j}$. (See \cite[Lemma 3.2.2]{Fe69} or \cite[Lemma 4]{Ki94}.)

If we write $Z' = Z \cup \bigcup_i g_i(Z_i)$ and $A_{i,j} = g_i(E_{i,j})$, then $A = Z' \cup \bigcup_{i,j} A_{i,j}$. Furthermore,
$$ \HH^Q(f(Z')) \leq \HH^Q(f(Z)) + \sum_i\HH^Q(f(g_i(Z_i)) = 0$$
and $f$ is bi-Lipschitz on each set $A_{i,j}$. This completes the proof.
\end{proof}

\begin{appendix}
\section{Proof of Proposition \ref{tangentsoftangents}}\label{appendix}

Our goal in this appendix is to prove Proposition \ref{tangentsoftangents}. We first observe a stronger statement in the case $\G=\RR^n$. This is a simple consequence of \cite[Proposition 3.1]{GCD14}. Note that, while that Proposition is stated for doubling metric measure spaces, it applies equally well to pointwise doubling metric measure spaces, using Lemma \ref{densitytangent}.

\begin{prop}\label{tangentsoftangentsRn}
Suppose $(X,d,\mu)$ is an Ahlfors regular metric space, $A\subset X$ is compact, and $f\colon X\rightarrow \RR^n$ is Lipschitz. Then, for $\mu$-almost every $x\in A$, for all $(Y,y,g)\in \Tan(A,x,f)$ and all $y'\in Y$, we have $\Tan(Y,y',g)\subseteq\Tan(A,x,f)$.
\end{prop}
\begin{proof}[Proof of Proposition \ref{tangentsoftangentsRn}]
Let $x\in A$ be a point at which the conclusion of \cite[Proposition 3.1]{GCD14} holds. Fix $(Y,y,g)\in\Tan(A,x,f)$, $y'\in Y$, and $(Z,z,h)\in\Tan(Y,y',g)$. We must show that
\begin{equation}\label{toft1}
 (Z,z,h)\in \Tan(A,x,f).
\end{equation}

We know that $(Z,z,h)\in\Tan(Y,y',g)$. Let $\eta_i\rightarrow 0$ be a sequence to which $(Z,z,h)$ is subordinate in $\Tan(Y,y',g)$.

Fix $\epsilon\in (0,\frac{1}{10})$. There is a choice $i=i(\epsilon)$ such that
\begin{equation}\label{toft2}
 D\left( \left(\eta_i^{-1} Y, y',\eta_i^{-1}(g-g(y))\right), (Z,z,h)\right) < \epsilon.
\end{equation}
We may also take $\eta_i<1$.

By \cite[Proposition 3.1]{GCD14}, we also have that
$$(Y,y',g-g(y'))\in\Tan(A,x,f).$$
Let $\lambda_j$ be a sequence to which $(Y,y',g-g(y'))$ is subordinate in $\Tan(A,x,f)$. We may select $j$, depending on our choices above, such that
$$ D\left( \left(\lambda_j^{-1} A, x,\lambda_j^{-1} (f-f(x))\right), (Y,y',g-g(y'))\right) < \eta_i \epsilon.$$
It follows from the definition of $D$ and the bound $\eta_i<1$ that
\begin{equation}\label{toft3}
D\left( \left((\eta_i\lambda_j)^{-1} A, x,(\eta_i\lambda_j)^{-1} (f-f(x))\right), \left(\eta_i^{-1}Y,y',\eta_i^{-1}(g-g(y'))\right)\right) <  \epsilon.
\end{equation}

Combining \eqref{toft2} with \eqref{toft3} and using the quasi-triangle inequality from Lemma \ref{Dproperties}\eqref{quasimetric}, we get
$$ D(\left((\eta_i\lambda_j)^{-1} A, x,(\eta_i\lambda_j)^{-1}(f-f(x))\right), (Z,z,h)) < 4\epsilon.$$
Since $\epsilon$ was arbitrary, this proves \eqref{toft1} and hence the proposition.

\end{proof}

In order to prove Proposition \ref{tangentsoftangents}, one could perhaps run a similar argument as in \cite[Proposition 3.1]{GCD14} but for Carnot group rather than Euclidean targets. However, the nonabelian nature of the group law in general Carnot groups makes some of the arguments more difficult. We therefore take a different tack here, reducing to the Euclidean statement, Proposition \ref{tangentsoftangentsRn}, by use of Assouad's embedding theorem \cite[Theorem 12.2]{He01}.

We begin with a preliminary lemma, which shows that the $\epsilon$-isometries in the definition of our Gromov--Hausdorff distance $D$ can be taken to be inverses up to small additive error. The proof is essentially identical to \cite[Lemma 2.5]{GCD14}.

\begin{lemma}\label{inverselemma}
Let $(X,x,f\colon X\rightarrow \G)$ and $(Y,y',g\colon Y\rightarrow \G)$ be packages, and let $\phi \colon X\rightarrow Y$ be an $\epsilon$-isometry such that
$$\sup_{B(x,1/\epsilon)} d_{cc}(f,g\circ\phi)<\epsilon.$$
Then there is a $C\epsilon$-isometry $\psi \colon Y\rightarrow X$ such that
$$\sup_{B(y,1/\epsilon)} d_{cc}(f\circ \psi,g)<C\epsilon$$
and
\begin{equation}\label{inverses}
 d(\psi \circ \phi (z), z) < C\epsilon \hspace{0.3cm} \text{ and } \hspace{0.3cm} d(\phi\circ\psi(w),w) <C\epsilon
\end{equation}
for all $z\in B(x,1/C\epsilon)$ and $w\in B(y,1/C\epsilon)$.

The constant $C$ depends only on the Lipschitz constants of $f$ and $g$.
\end{lemma}
\begin{proof}
For simplicity, we denote the metrics on $X$ and $Y$ both by $d$. Let $N\subset B(x,1/\epsilon)$ be a maximal separated $\epsilon$-net. In other words,
$$ d(y,z) \geq \epsilon $$
if $y,z\in N$ and $y\neq z$, and
$$ \dist(z,N)<\epsilon $$
for all $z\in B(x,1/\epsilon)$. We can also arrange that $x\in N$.

The fact that $\phi$ is an $\epsilon$-isometry implies that $\phi|_N$ is injective. Let $N'=\phi(N)\subset Y$. Because $\phi$ is an $\epsilon$-isometry, we know that every point of $B(y,1/2\epsilon)$ is within $3\epsilon$ of a point in $N'$.

Let $\pi\colon Y\rightarrow N'$ denote any choice of closest-point projection, i.e., $\pi(Y)\subset N'$ and $d(y,\pi(y)) = \dist(y,N')$. Then $\pi$ preserves distances up to an additive error of $6\epsilon$ for points in $B(y,1/2\epsilon)$. Let
$$ \psi = \left(\phi|_N\right)^{-1} \circ \pi \colon Y \rightarrow X. $$

We first claim that $\psi$ is a $7\epsilon$-isometry. Fix $y_1, y_2\in B(y,1/7\epsilon)$. We have
\begin{align*}
|d(\psi(y_1), \psi(y_2)) - d(y_1, y_2)| &\leq |d(\phi^{-1}(\pi(y_1)), \phi^{-1}(\pi(y_2))) - d(\pi(y_1), \pi(y_2))|\\
&+ |d(\pi(y_1), \pi(y_2)) - d(y_1, y_2)|\\
&\leq \epsilon + 6\epsilon\\
&= 7\epsilon.
\end{align*}

In addition, for $r\leq 1/(7\epsilon)$,
$$\psi(B(y,r))\supseteq N \cap B(x, r-\epsilon) $$
and therefore
$$ N_{7\epsilon}(\psi(B(y,r))) \supseteq B(x,r-7\epsilon). $$

We now claim that 
$$\sup_{B(y,1/7\epsilon)} d_{cc}(g , f\circ \psi) < C\epsilon,$$
where $C$ depends only on the Lipschitz constant of $g$. For $z\in B(y, 1/7\epsilon)$, we have
\begin{align*}
d_{cc}(g(z) , f(\psi(z))) &= d_{cc}(g(z) , f((\phi|_N)^{-1}(\pi(z))))\\
&\leq d_{cc}(g(z) , g(\pi(z))) + d_{cc}((g(\pi(z)) , f((\phi|_N)^{-1}(\pi(z))))\\
&\leq 6\epsilon \LIP(g) + \epsilon.
\end{align*}

Finally, for any $z\in B(x,1/\epsilon)$, we have
\begin{align*}
d(z,\psi(\phi(z))) &\leq d(z,\pi(z)) + d(\pi(z),\psi(\phi(\pi(z)))) + d(\psi(\phi(\pi(z))),  \psi(\phi(z)))\\
&\leq \epsilon+0 + 3\epsilon\\
&=4\epsilon 
\end{align*}

A similar argument using the $3\epsilon$-density of $N'$ in $B(y,1/2\epsilon)$ shows that
$$ d(w,\phi(\psi(w))) \leq 10\epsilon.$$

This completes the proof.
\end{proof}

To prove Proposition \ref{tangentsoftangents}, we will reduce to Proposition \ref{tangentsoftangentsRn} by embedding the Carnot group $\G$ in some $\RR^N$. This cannot be done in a bi-Lipschitz way, but it can be done in a ``snowflake'' way; i.e., there is a bi-Lipschitz embedding of $(\G, d_{cc}^{1/2})$ in some $\RR^N$. With the following lemma, this will suffice for our purposes.

\begin{lemma}\label{snowflaketangent}
Let $\G$ be a Carnot group. Let $(X,x,f \colon X\rightarrow\G)$ be a metrically doubling Lipschitz package, and let $e \colon (\G,d_{cc}^{1/2})\rightarrow \RR^N$ be a bi-Lipschitz embedding. 

Write $\tilde{X} = (X,d^{1/2})$ and $\tilde{f}=e\circ f$, so that $\tilde{f}$ is Lipschitz on $\tilde{X}$. Let
$$ (Y,y,g)\in \Tan(X,x,f),$$
$$ (\tilde{Y},\tilde{y},\tilde{g}) \in \Tan(\tilde{X},x,\tilde{f}),$$
be subordinate to the sequences $\lambda_n \rightarrow 0$ and $\lambda_n^{1/2} \rightarrow 0$, respectively.

Then $(\tilde{Y},\tilde{y},\tilde{g})$ is pointedly isometric to $( (Y,d^{1/2}),y,\hat{e}\circ g))$ for some bi-Lipschitz embedding $\hat{e}\colon (\G,d^{1/2})\rightarrow \RR^N$.
\end{lemma}
\begin{proof}
By passing to subsequences, we may also achieve the local uniform convergence of the sequence of uniformly bi-Lipschitz embeddings
$$ p\mapsto \lambda_n^{-1/2}( e(f(x)\cdot \delta_{\lambda_n}(p)) - e(f(x))) \colon (\G, d_{cc}^{1/2})\rightarrow \RR^N.$$
to a bi-Lipschitz embedding $\hat{e} \colon (\G, d_{cc}^{1/2})\rightarrow \RR^N$.

Now, let
\begin{equation}\label{Xconvergence}
 \epsilon_n = D\left( (\lambda_n^{-1} X, x, \delta_{\lambda_n^{-1}}(f(x)^{-1}f)), (Y,y,g) \right) \rightarrow 0.
\end{equation}
\begin{equation}\label{tildeconvergence}
 \tilde{\epsilon}_n = D\left( (\lambda_n^{-1/2} \tilde{X}, x, \lambda_n^{-1/2}(\tilde{f} - \tilde{f}(x))) , (\tilde{Y},\tilde{y},\tilde{g}) \right) \rightarrow 0.
\end{equation}

Let
$$ \phi_n \colon \lambda_n^{-1}X\rightarrow Y \hspace{0.3cm} \text{and} \hspace{0.3cm} \psi_n \colon Y\rightarrow \lambda_n^{-1}X$$
be $\epsilon_n$-isometries achieving the distances in \eqref{Xconvergence}, and similarly let
$$ \tilde{\phi}_n \colon \lambda_n^{-1/2}\tilde{X}\rightarrow \tilde{Y} \hspace{0.3cm} \text{and} \hspace{0.3cm} \psi_n \colon \tilde{Y}\rightarrow \lambda_n^{-1}X$$
be $\tilde{\epsilon}_n$-isometries achieving the distances in \eqref{tildeconvergence}. Furthermore, we may assume that $\phi_n$ and  $\psi_n$ are approximate $C\epsilon_n$-inverses, in the sense of \eqref{inverses}.

For convenience, we labeled the domains of these mappings as $X,Y,\tilde{X},\tilde{Y}$, but really they are defined only on balls of radius $\epsilon_n^{-1}$ or $\tilde{\epsilon}_n^{-1}$ centered at the relevant base points.

Let $R_n = \frac{1}{2C}\min(\epsilon_n^{-1/2}, \tilde{\epsilon}_n^{-1/2})$, where $C$ is the constant from Lemma \ref{inverselemma}. Let $i_n$ be the mapping 
$$ i_n = \phi_n \circ \tilde{\psi}_n \colon B_{\tilde{Y}}(\tilde{y}, R_n) \rightarrow Y.$$
For any $z,w\in B_{\tilde{Y}}(\tilde{y},R_n)$, we have
\begin{align*}
|d_Y(i_n(z), i_n(w)) - d_{\tilde{Y}}(z,w)^2| &\leq |d_{\lambda^{-1}X}(\tilde{\psi}_n(z), \tilde{\psi}_n(w)) - d_{\tilde{Y}}(z,w)^2| + \epsilon_n\\
&= |d_{\lambda^{-1/2}\tilde{X}}(\tilde{\psi}_n(z), \tilde{\psi}_n(w))^2 - d_{\tilde{Y}}(z,w)^2| + \epsilon_n\\
&\leq |d_{\lambda^{-1/2}\tilde{X}}(\tilde{\psi}_n(z), \tilde{\psi}_n(w)) - d_{\tilde{Y}}(z,w)|(3R_n) + \epsilon_n\\
&\leq \tilde{\epsilon}_n(3R_n) + \epsilon_n\\
&\leq 3\tilde{\epsilon}_n^{1/2} + \epsilon_n
\end{align*}
which tends to zero as $n$ tends to infinity.

It follows that $i_n$ converges to an isometric embedding $i$ of $(\tilde{Y},d^{1/2})$ into $(Y,d)$.  (Note that both $Y$ and $\tilde{Y}$, being complete doubling spaces, have the property that closed balls are compact, which is all that is needed here.)

Furthermore, it is easy to see using properties \eqref{isom1} and \eqref{isom2} in Definition \ref{epsilonisomdef}, applied to $\phi_n$ and $\tilde{\psi}_n$ that $i$ is surjective with $i(\tilde{y})=y$.

It remains to show that
$$ \hat{e} \circ g \circ i = \tilde{g}.$$

Fix $\epsilon>0$ and $z\in Y$. Choose $n$ large so that $\epsilon_n, \tilde{\epsilon}_n<\epsilon$ and in addition
$$ |\hat{e}(p) - \lambda_n^{-1/2}( e(f(x)\cdot \delta_{\lambda_n}(p)) - e(f(x)))| < \epsilon$$
uniformly for $p$ in a ball of radius $1$ around $g(i(z))$.

To avoid cumbersome constants below, we use the notation $O(\epsilon)$ to denote a quantity which is bounded by $C\epsilon$ for some constant $C$ depending only on the Lipschitz constants of $f$,$g$, and $e$, which may change from line to line. Then
\begin{align*}
|\hat{e}\circ g \circ i(z) - \tilde{g}(z)|&\leq |\lambda_n^{-1/2}( e(f(x)\cdot \delta_{\lambda_n}(g(i(z)))) - e(f(x)))\\
&\phantom{\leq } - \lambda_n^{-1/2}( e(f(\tilde{\psi}_n(z))) - e(f(x)))| + O(\epsilon)\\
&\leq |\lambda_n^{-1/2}( e(f(x)\cdot f(x)^{-1} \cdot f(\psi_n(i(z)))) - e(f(x)))\\
&\phantom{\leq} - \lambda_n^{-1/2}( e(f(\tilde{\psi}_n(z))) - e(f(x)))| + O(\epsilon) \\
&\leq |\lambda_n^{-1/2}( e(f(\psi_n(\phi_n(\tilde{\psi}_n(z))))) - e(f(x)))\\
&\phantom{\leq} - \lambda_n^{-1/2}( e(f(\tilde{\psi}_n(z))) - e(f(x)))| + O(\epsilon)\\
&= O(\epsilon),
\end{align*}
where in the last line we used \eqref{inverses} and the Lipschitz properties of $e$ and $f$.

As $\epsilon$ was arbitrary, this shows that $\hat{e}\circ g \circ i = \tilde{g}$. Hence, $i$ is a pointed isometry between the packages $(\tilde{Y}, \tilde{y},\tilde{g})$ and $((Y,d^{1/2}), y, \hat{e}\circ g)$.

\end{proof}

\begin{proof}[Proof of Proposition \ref{tangentsoftangents}]
Fix a bi-Lipschitz embedding $e\colon (\G,d_{cc}^{1/2})\rightarrow \RR^n$ as in Lemma \ref{snowflaketangent}, using Assouad's embedding theorem \cite[Theorem 12.2]{He01}. Write $\tilde{A} = (A,d^{1/2})$ and $\tilde{f}=e\circ f$, so that $\tilde{f}$ is Lipschitz on $\tilde{A}$.

Choose $x\in A$ such that the conclusion of Proposition \ref{tangentsoftangentsRn} holds for the package $(\tilde{A}, x, \tilde{f}\colon \tilde{A}\rightarrow\RR^N)$. Let $(Y,y,g)\in\Tan(A,x,f)$, subordinate to the sequence $\lambda_n \rightarrow 0$, and fix $y' \in Y$. Let $(Z,z,h)\in\Tan(Y,y',g)$ be subordinate to the sequence $\eta_n \rightarrow 0$, and suppose that $h$ is bi-Lipschitz. 

By passing to a subsequence, we may also find a tangent
$$ (\tilde{Y},\tilde{y},\tilde{g}) \in \Tan(\tilde{A},x,\tilde{f}),$$
subordinate to the sequence $\lambda_n^{1/2}$.

By Lemma \ref{snowflaketangent}, we know that $(\tilde{Y},\tilde{y},\tilde{g})$ is pointedly isometric to $( (Y,d^{1/2}),y,\hat{e}\circ g))$ for some bi-Lipschitz embedding $\hat{e}$.

Passing to a further subsequence, we also have
$$ (\tilde{Z},\tilde{z},\tilde{h}) \in \Tan(\tilde{Y},y',\tilde{g}),$$
subordinate to the sequence $\eta_n^{1/2}$. Appealing to Lemma \ref{snowflaketangent} again, $(\tilde{Z},\tilde{z},\tilde{h})$ is pointedly isometric to 
$$((Z,d^{1/2}), z, e'\circ h),$$
for some bi-Lipschitz embedding $e'$.

Now, we know that
$$(\tilde{Z},\tilde{z},\tilde{h}) = ((Z,d^{1/2}), z, e'\circ h) \in \Tan(\tilde{A},x,\tilde{f})$$
by Proposition \ref{tangentsoftangentsRn}. Let $\alpha_n \rightarrow 0$ be a sequence to which this tangent is subordinate. By passing to a subsequence, we may assume that $(A,x,f)$ has a tangent $(W,w,j)$ subordinate to the sequence $\alpha^2_n$.

Once again, by Lemma \ref{snowflaketangent}, it follows that
$$ (\tilde{Z},\tilde{z},\tilde{h}) = ((W,d^{1/2}),w,b\circ j)$$
for some bi-Lipschitz embedding $b \colon (\G,d_{cc}^{1/2})\rightarrow \RR^N$. We conclude by observing that $j = (b|_{b(\G)})^{-1} \circ e' \circ h$ is bi-Lipschitz from $W$ to $(\G,d_{cc})$.
\end{proof}

\end{appendix}

\bibliographystyle{plain}
\bibliography{carnot2}

\end{document}